\documentclass[11pt, a4paper]{article}
\usepackage[utf8]{inputenc}
\usepackage{amsmath,amssymb,amsfonts,dsfont}
\usepackage{amsthm}
\usepackage{mathtools}
\usepackage{verbatim}
\usepackage{a4wide}
\usepackage{epsfig,epic,eepic,graphicx}
\usepackage[mathcal]{euscript}
\usepackage{import}
\usepackage{xifthen}

\usepackage{marginnote} 
\reversemarginpar

\usepackage[usenames,dvipsnames]{xcolor}

\usepackage[unicode, bookmarks, colorlinks, breaklinks, pagebackref,
	pdfauthor={Trien-Thuyen Dang},
]{hyperref}  
\hypersetup{linkcolor=OliveGreen,citecolor=OliveGreen,filecolor=black,urlcolor=Magenta}

\usepackage[capitalize,nameinlink,noabbrev]{cleveref}

\usepackage{cite}

\usepackage{enumitem}

\newtheorem{theorem}{Theorem}
\newtheorem{proposition}[theorem]{Proposition}%
\newtheorem{remark}{Remark}%

\newtheorem{definition}{Definition}%

\newtheorem{lemma}[theorem]{Lemma}
\numberwithin{equation}{section}

\newcommand{\leb}[1]{\mathop \le \limits_{(#1)}}

\newcommand{\coex}[2][\hphantom{longexplanation}]{%
  \underset{{#1}}{\overset{\hphantom{longexplanation}}{#2}}
}

\newcommand{\hk}{\hspace*{.12in}}

\def\beq{\begin{eqnarray*}}\def\eeq{\end{eqnarray*}}
\def\bq{\begin{equation}}\def\eq{\end{equation}}

\newcommand{\vuong}{\square}

\newcommand{\NN}{{\mathbb N}}

\newcommand{\RR}{{\mathbb R}}

\newcommand{\email}[1]{\href{mailto:#1}{\texttt{#1}}}

\newcommand{\abs}[1]{\left\lvert #1\right\rvert}

\newcommand{\norm}[1]{\left\lVert #1\right\rVert}

\newcommand{\cv}[1][]{%
\ifthenelse{\isempty{#1}}{\xrightarrow[\hphantom{~2~}]{}}{\xrightarrow[\hphantom{~2~}]{#1}}%
}
\newcommand{\wcv}[1][]{%
\ifthenelse{\isempty{#1}}{\xrightharpoonup[\hphantom{~2~}]{}}{\xrightharpoonup[\hphantom{~2~}]{#1}}%
}

\DeclareMathOperator{\supp}{supp}

\DeclareMathOperator{\sign}{{sign}}
\DeclareMathOperator{\dist}{{dist}}

\def\XXint#1#2#3{{\setbox0=\hbox{$#1{#2#3}{\int}$ }
\vcenter{\hbox{$#2#3$ }}\kern-.6\wd0}}

\DeclarePairedDelimiterX\Set[1]\{\}{%
  #1%
}
\allowdisplaybreaks

\begin{document}

\title{Global boundedness of solutions of degenerate and
  non-uniform parabolic equations}

\author{ Thuyen Dang\footnote{
    Department of Statistics/Committee on
    Computational and Applied Mathematics, University of Chicago, 5747
    S. Ellis Avenue, Chicago, Illinois 60637, USA
    (\email{thuyend@uchicago.edu} (preferred email)),
  } \footnote{Department of Mathematics and Statistics, Purdue University Northwest, 2200 169th Street, Hammond, Indiana 46323, USA
  (\email{ttdang@purdue.edu}).}
\and
  Duong Minh Duc\footnote{
    Department of Mathematics and Computer Science, Vietnam
    National University Ho Chi Minh City--University of Science, 227
    Nguyen Van Cu Street, Phuong Cho Quan, Ho Chi Minh City, Vietnam
  (\email{dmduc@hcmus.edu.vn}).}
}
\date{}

\maketitle

\begin{abstract}
  Let $2 \le N\in\NN$, $\Omega$ be a bounded open in $\RR^{N}$,
  $T\in (0,\infty)$, $Q=\Omega\times (0,T)$, $u$ be a weak solution of
  parabolic equation
  $\displaystyle \frac{\partial u}{\partial t} -Lu= f$, where $L$ is
  an elliptic operator on a space of functions on $Q$. The
  coefficients of $L$ may be not bounded, not strictly nor uniformly
  elliptic, and not of Muckenhoupt type. We obtain global boundedness
  of $u$.  Our result can be applied to $u$, which may vanish on
  $(A\times (0,T))\cup (\Omega\times \{0\})$ of the boundary of $Q$
  and is free outside this set.
\end{abstract}



\section{Introduction}

Let 
\begin{itemize}
\item $\Omega$ be a bounded open subset of the Euclidean space $\RR^{N}$
with $2 \le N \in \NN$ and $\partial\Omega$ be its boundary;
\item  $A$ be a
subset of $\partial\Omega$ with non-zero (surface) measure;
\item $T\in (0,\infty)$ and  $Q\coloneqq\Omega\times (0,T)$.
\item $\bar{t}\in \left(\frac{2N^{2}+2N-2}{N^{2}+2N-1},2 \right)$,
  $r\coloneqq \frac{\bar{t}(N+1)-2}{N-\bar{t}}>2$,
  $\overline{r}\in (2,r)$;
\item $b_{ij}$ be measurable functions on $\Omega$ for every
  $i,j \in \{1,\ldots,N\}$ such that  $b_{ij}= b_{ji}$, and
  $B\coloneqq\{b_{ij}\}_{1 \le i,j \le N}$;
\item  $b$ and
  $\overline{b}$ be non-negative measurable functions on $\Omega$.
\end{itemize}

We consider following conditions 
\begin{subequations}
  \label{eq:ccc}
  \begin{align}
  &b(x)|\xi|^{2}\le b_{ij}(x)\xi_{i}\xi_{j}\le \overline{b}(x)\abs{\xi}^{2}\qquad\forall~\xi=(\xi_{1},\cdots,\xi_{N})\in \RR^{N},~ x\in \Omega,
\label{c1}\\
&b^{-1}\in L^{\frac{\overline{t}}{2-\overline{t}}}(\Omega),
\label{c3}\\
&\overline{b}\in L^{\frac{r}{r-2}}(\Omega).
\label{c6}
\end{align}
\end{subequations}
Define
\begin{equation}
  Lu \coloneqq - \sum_{i,j=1}^{N}\frac{\partial}{\partial x_{i}}
  \left(b_{ij} \frac{\partial u}{\partial x_{j}} \right).
\label{c5}
\end{equation}

Let $W_{B,A,T}(Q)$ and $V_{B,A,T}(Q)$ be as in Definition \ref{d44}.
If $B=\{\delta^{i}_{j}\}_{i,j}$ and $A = \partial\Omega$, here
$\delta_j^i$ is the Kronecker notation, then
$V_{B,A,T}(\Omega)$ contains the usual Sobolev's space
$W^{1,2}_{0}(Q)$.

Let $f\in L^{1}(Q)$ and $u$ be in $V_{B,A,T}(Q)$. We say
$\displaystyle \frac{\partial u}{\partial t} -Lu=f$ in the weak sense,
if
\begin{equation}
  \int_{Q} \left(\frac{\partial u}{\partial t}\varphi+
    \sum_{i,j=1}^{N}b_{ij}\frac{\partial u}{\partial x_{i}}\frac{\partial
      \varphi}{\partial x_{j}} \right)dz = \int_{Q}f\varphi\, dz,
  \qquad \forall~\varphi \in W_{B,A,T}(Q).
 \label{eq}
\end{equation}

In this paper we obtain the boundedness of $u$ as follows

\begin{theorem}[Main theorem]
  Let $A$ be  admissible with respect to $\Omega$ and assume
\eqref{eq:ccc} holds. 
Let $0 < a\in L^{\frac{\overline{r}}{\overline{r}-2}}(Q)$, and $a_{0}$, $a_{1}$, $a_{2}$ be
non-negative measurable functions on $Q$ such that
\begin{equation}
b^{-1} a_{0}^{2}+a_{1}+ a_{2}\le a,
\label{th1:cc8}
\end{equation}
and
\begin{equation}
 \abs{f\varphi}\le \left(a_{0}\abs{\nabla_{x} u}+a_{1}\abs{u}  +a_{2} \right)\abs{\varphi},\qquad\forall~\varphi\in W_{B,A,T}(Q).
\label{elz},
\end{equation}
where $u\in V_{B,A,T}(Q)$ is the weak solution of \eqref{eq}.

Then there exist $\delta > 0$ and $c\left(r,\alpha,Q,\norm{u}_{L^{\alpha}(Q)} \right) > 0$ such that
whenever $u \in  L^{\alpha}(Q)$
with some $\alpha > \bar{r}\left( \frac{2}{3}-\delta
  \right)$, we have $u \in L^{\infty}(Q)$ and
\begin{equation}
  \norm{u}_{L^{\infty}(Q)}\le c \left(r,\alpha,\Omega,\norm{u}_{L^{\alpha}(Q)} \right). \label{th1:le460abc}
 \end{equation}

\label{th1}
\end{theorem}

By Remarks \ref{r4} and \ref{r4b} in Section \ref{bo}, our result applies to the case that $\overline{b}$ is not in the Muckenhoupt
class $A_{\frac{N+2}{N}}$, and $L$ is degenerate and non-uniformly
elliptic. In \cite{ducGlobalHolderContinuity}, we obtained the similar
results for elliptic equations. If $b$ and $\overline{b}$ are constants
and $A=\partial Q$, then Theorem \ref{th1} was proved in
\cite{ladyzenskajaLinearQuasilinearEquations1968}. We refer the
readers to the following papers and references therein:
\cite{defilippisSketchesNonuniformlyElliptic2025,defilippisNonuniformlyEllipticSchauder2023}--for the regularity theory of non-uniform equations, and
\cite{dongNondivergenceFormDegenerate2024,dongDegenerateLinearParabolic2023}--for degenerate parabolic equations. Our theoretical result is
motivated by the model of high-contrast materials, see
\cite{armstrongRenormalizationGroupElliptic2025} for state of the art
of this problem in homogenization setting.

The plan of our paper is as follows.  In Sections \ref{wi} and
\ref{app}, we introduce and study the properties of function spaces,
which are used in the proof of Theorem \ref{th1}. We shall use the
techniques of Moser in \cite{moserHarnackInequalityParabolic1964}, but
these techniques are only applicable to local regularity. Therefore we
have to modify auxiliary functions in
\cite{moserHarnackInequalityParabolic1964} in Section \ref{aux}.  In
Section \ref{bo} we obtain the global boundedness of solutions of
parabolic equations, in particular, we will prove Proposition
\ref{pro46}, which implies Theorem \ref{th1}.

\section{Function spaces}\label{wi}

In this section, let $N\ge 2$, $\Omega$ be a bounded open in
$\RR^{N}$, $\bar{t}\in \left(\frac{2N^{2}+2N-2}{N^{2}+2N-1},2 \right)$,
$r=\frac{\bar{t}(N+1)-2}{N-\bar{t}}$, $\bar{t}^{\ast}=\frac{\bar{t}N}{N-\bar{t}}$. Suppose $b_{ij}$,
$B$, $b$, and $\bar{b}$ satisfy \eqref{eq:ccc}. We have
\begin{align*}
  r
  = \frac{\bar{t}N-(2-\bar{t})}{N-\bar{t}}< \bar{t}^{\ast},
\end{align*}
and
\begin{align*}
  r-2
  &\coex{=} \frac{\bar{t}N-(2-\bar{t}) -2(N-\bar{t})}{N-\bar{t}}\\
  &\coex{=} \frac{\bar{t}(N+3) -2(N+1)}{N-\bar{t}}\\
&\coex[\left(\bar{t}>\frac{2N^{2}+2N-2}{N^{2}+2N-1}\right)]{>} \frac{2}{(N-\bar{t})(N^{2}+2N-1)}[(N^{2}+N-1)(N+3)-(N+1)(N^{2}+2N-1)]\\
  &\coex{=}\frac{2(N^{2}+N-2)}{(N-\bar{t})(N^{2}+2N-1)}\\
  &\coex[(N\ge 2)]{>} 0.
\end{align*}
Thus
\begin{equation}
  1< t<2 < r <t^{\ast}.
 \label{rt}
\end{equation}

\begin{definition}
  Define
  \begin{itemize}
  \item $\displaystyle \nabla_{x}v(x,t)\coloneqq\left(\frac{\partial v}{\partial
    x_{1}}(x,t),\cdots,\frac{\partial v}{\partial x_{N}}(x,t)
\right)$.
\item $\displaystyle\nabla v(x,t)\coloneqq\left(\frac{\partial v}{\partial
    x_{1}}(x,t),\cdots,\frac{\partial v}{\partial x_{N+1}}(x,t) \right)$
  where
  $\displaystyle \frac{\partial v}{\partial x_{N+1}}\coloneqq\frac{\partial
    v}{\partial t}$.
\item $C^{1}(\Omega,A)$:~ the family of function
  $v|_{\Omega}$ with $v$ in $C_{c}^{1}(\RR^{N})$ such that
  $ A\cap \supp (v) = \varnothing$.
\item $C^{1}(Q,A)$:~ the family of
  function $w\vert_{Q}$ with $w$ in $C_{c}^{1}(\RR^{N+1})$ such that
  \begin{align*}
  \supp(w)\cap [(A\times [0,T])\cup (\overline{\Omega}\times \left\{0\right\})]
  =\varnothing.
  \end{align*}
\item  $\displaystyle\norm{v}_{C,B} \coloneqq
  \left\{\int_{\Omega}\sum_{i,j=1}^{N}\frac{\partial v}{\partial
    x_{i}}\frac{\partial v}{\partial x_{j}}b_{ij}dx\right\}^{\frac{1}{2}}
  \qquad\forall~v \in C^{1}(\Omega,A).$
\item  $\displaystyle\norm{v}_{C,b} \coloneqq \left\{\int_{\Omega}|\nabla _{x}
  v|^{2}bdx\right\}^{\frac{1}{2}} \qquad\forall~v \in
C^{1}(\Omega,A).$
\item  $\displaystyle\norm{v}_{C,B,T} \coloneqq
  \left\{\int_{0}^{T}\int_{\Omega}\sum_{i,j=1}^{N}\frac{\partial
    v}{\partial x_{i}}(x,t)\frac{\partial v}{\partial
    x_{j}}(x,t)b_{ij}(x)dxdt\right\}^{\frac{1}{2}} \qquad\forall~v \in
  C^{1}(Q,A).$
\item  $\displaystyle\norm{v}_{C,b,T} \coloneqq
  \left\{\int_{0}^{T}\int_{\Omega}|\nabla_{x}v|^{2}(x,t)
  b(x)dxdt\right\}^{\frac{1}{2}} \qquad\forall~v \in C^{1}(Q,A).$
\item  $\displaystyle\norm{v}_{C,V,B,T} \coloneqq \norm{\frac{\partial v}{\partial  t}}_{L^{2}(Q)} +\norm{v}_{C,B,T} \qquad\forall~v \in C^{1}(Q,A).$
  \end{itemize}
    \label{space}
 \end{definition}
  We have
  \begin{align}
    \norm{w}_{C,b}
    &\coex[\eqref{c1}]{\le} \norm{w}_{C,B}\qquad\forall~w \in C^{1}(\Omega,A),
  \label{i1}\\
  \norm{w}_{C,b,T}
  &\coex[\eqref{c1}]{\le} \norm{w}_{C,B,T}\qquad\forall~w \in C^{1}(Q,A).
  \label{i1b}
\end{align} 
  \begin{definition} Let  $A\subset \partial\Omega$. We say ~$A$   is admissible with respect to $\Omega$, if there is a positive real number $C(r,\Omega,A)>0$ such that
  \begin{equation}\left\{\int_{\Omega} |v|^{r}dx\right\}^{\frac{1}{r}}\le C(r,\Omega,A)\left\{\int_{\Omega} |\nabla v|^{\overline{t}}dx\right\}^{\frac{1}{\overline{t}}} \qquad\forall ~v \in C^{1}(\Omega,A).
\label{ineq}
\end{equation}   
  \label{d41}
 \end{definition} \hk   
 We have the following result.
 Let $\Sigma_{N-1}$ be the unit sphere of $\RR^{N}$. We denote the canonical measure on $\Sigma_{N-1}$ by $\sigma$ (see  \cite[\S III.3]{chavelRiemannianGeometry2006}). We have the following examples of admissible subsets.
  \begin{proposition}  $A$ is admissible with respect to $\Omega$ in
    following cases. 
    \begin{itemize}[wide]
    \item[(i)] $\Omega$ is convex and $A=\partial\Omega$.
    \item[(ii)] $\Omega=(0,1)^{N}$ and $A =\partial\Omega \setminus
      (\{0\}\times[0,1]^{N-1})$.
    \item[(iii)] $\Omega$ is connected and of class $C^{1}$ and $\mu(A)>0$, where $\mu$ is the Riemannian measure on $\partial\Omega$ (see definition in \cite[pp.120-121]{chavelRiemannianGeometry2006}).
    \end{itemize}
 \label{admissible}
  \end{proposition}
  \begin{proof}
 Let $u\in C^{1}(\Omega,A)$. Fix $x\in \Omega$, $\omega\in \Sigma_{N-1}$ and $r_{\omega}\in(0,\infty)$ such that $x'=x+r_{\omega}\omega\in  A$ and $x+sw\in \Omega$ for every $s\in [0,r_{x})$. We have
 \begin{align}
   \label{az11}
   \begin{split}
     |u(x)|
   &=\abs{u(x)-u(x+r_{\omega}\omega)}\\
   &=\abs{\int_{0}^{r_{\omega}}\frac{du(x+\xi\omega)}{d\xi}d\xi}\\
   &= \abs{\int_{0}^{r_{\omega}}\nabla u(x+\xi\omega)\cdot \omega d\xi}\\
   &\le  \int_{0}^{r_{\omega}} \abs{\nabla
     u(x+\xi\omega)}d\xi\\
   &=\int_{0}^{r_{\omega}}\frac{|\nabla
  u(x+\xi\omega)|}{|x+\xi\omega -x|^{N-1}}\xi^{N-1}d\xi.
   \end{split}
\end{align}

\begin{itemize}[wide]
\item[$(i)$] Since $\Omega$ is convex, we have
 \[\Omega\setminus \{x\}= \bigcup_{\omega\in\Sigma_{N-1}}\{x+sw: s\in
   (0,r_{\omega})\}.\]
 Using the polar coordinate with the pole at $x$, we get 
 \begin{align}
   \label{a11}
   \begin{split}
     |u(x)|
     &\coex{=}\frac{1}{\sigma(\Sigma_{N-1})}\int_{\Sigma_{N-1}}|u(x)|d\omega\\ &\coex[\eqref{az11}]{\le}\frac{1}{\sigma(\Sigma_{N-1})}\int_{\Sigma_{N-1}}\int_{0}^{r_{\omega}}\frac{|\nabla
     u(x+\xi\omega)|\xi^{N-1}}{|x+\xi\omega -x|^{N-1}}d\xi
                                                                          d\omega\\
     &\coex{=}\frac{1}{\sigma(\Sigma_{N-1})}\int_{\Omega}\frac{|\nabla
     u(y)|}{|y -x|^{N-1}}dy.
   \end{split}
\end{align}
By Hardy–Littlewood–Sobolev Theorem in \cite[p.35]{linaresIntroductionNonlinearDispersive2015} and \eqref{a11}, there is a positive constant $c(N,\bar{t})$ such that
\begin{align*}
\left\{\int_{\Omega}
  |u|^{r}dx\right\}^{\frac{1}{r}}
  &\coex[\text{H\"{o}lder}]{\le}
  |\Omega|^{\frac{\bar{t}^{\ast}-r}{r\bar{t}^{\ast}}} \left\{\int_{\Omega}
    |u|^{\bar{t}^{\ast}}dx\right\}^{\frac{1}{\bar{t}^{\ast}}}\\
  &\coex[\substack{\text{H–L-S Theorem}\\ \text{and } \eqref{a11}}]{\le} \frac{c(N,\bar{t})|\Omega|^{\frac{\bar{t}^{\ast}-r}{r\bar{t}^{\ast}}} }{\sigma(\Sigma_{N-1})}\left\{\int_{\Omega} |\nabla u|^{\bar{t}}dx\right\}^{\frac{1}{\bar{t}}}.
\end{align*}
Thus  we get \eqref{ineq}.
\item[$(ii)$]   Put $\overline{\Sigma}_{N-1} =\{(w_{1},\ldots,w_{N})\in \Sigma_{N-1}: w_{1}>0\}$. Let $x\in \Omega$ and $w= (w_{1},\ldots, w_{N})
    \in \overline{\Sigma}_{N-1}$. Then $w_{1}>0$. Put $r_{w}$ as in
    the proof of $(i)$. Since $x+r_{w}w= (x_{1}+r_{w}w_{1},\ldots,
    x_{N}+r_{w}w_{N}$, $x_{1}>0$ and $w_{1}>0$, we have
    $x_{1}+r_{w}w_{1}>0$ and  $x+r_{w}w\in \partial\Omega \setminus
    (\{0\}\times[0,1]^{N-1})=A$.
    
    Put
 \[ X= \bigcup_{w\in \overline{\Sigma}_{N-1}} \{x+sw: s \in (0,r_{w})\}.\]  
  Then $X$ is open subset of $\Omega$. Thus we can use the polar coordinate with the pole at $x$ and get   
  \begin{align}
    \label{a11b}
    \begin{split}
      |u(x)|
      &=\frac{1}{\sigma(\overline{\Sigma}_{N-1})}\int_{\overline{\Sigma}_{N-1}}|u(x)|d\omega\\
      &\leb{\eqref{az11}}\frac{1}{\sigma(\overline{\Sigma}_{N-1})}\int_{\overline{\Sigma}_{N-1}}\int_{0}^{r_{\omega}}\frac{|\nabla u(x+\xi\omega)|\xi^{N-1}}{|x+\xi\omega -x|^{N-1}}d\xi d\omega\\  &=\frac{1}{\sigma(\overline{\Sigma}_{N-1})}\int_{X}\frac{|\nabla
        u(y)|}{|y -x|^{N-1}}dy\\ &\le\frac{1}{\sigma(\overline{\Sigma}_{N-1})}\int_{\Omega}\frac{|\nabla
        u(y)|}{|y -x|^{N-1}}dy.
    \end{split}
\end{align}
Arguing as in the proof of $(i)$, we get \eqref{ineq}.

\item[$(iii)$]  We shall prove
  \begin{equation}
    \left\{\int_{\Omega} |v|^{r}dx\right\}^{\frac{1}{r}}\le
    C(r,\Omega,A)\left\{\int_{\Omega} |\nabla
      v|^{\bar{t}}dx\right\}^{\frac{1}{\bar{t}}}\qquad\forall ~v \in
    W^{1,\bar{t}}(\Omega),~ Bv=0 \text{ on } A,
\label{ineqb}
\end{equation}   
where $W^{1,\bar{t}}(\Omega)$ is the usual Sobolev space and $Bv$ is the
trace of $v$ on $\partial\Omega$.

 Assume by contradition that there is a sequence $\{v_{n}\}$ in
 $W^{1,\bar{t}}(\Omega)$ such that $Bv_{n}$ = 0 on $A$,
 $\norm{v_{n}}_{L^{r}(\Omega)}=1$ and $\norm{\nabla
 v_{n}}_{L^{\bar{t}}(\Omega)}\le \frac{1}{n}$ for every $n\in \NN$.
 
Note that
\begin{equation}
\frac{1}{\bar{t}}-\frac{1}{N-1}\frac{\bar{t}-1}{\bar{t}}=\frac{N-\bar{t}}{(N-1)\bar{t}} \coex[(\bar{t}>1)]{<} \frac{1}{\bar{t}}.
\label{ineqc}
\end{equation}\hk
By \eqref{rt}, $\{v_{n}\}$ is bounded in $W^{1,\bar{t}}(\Omega)$. By
Theorems 3.18 and 9.1 in \cite{brezisFunctionalAnalysisSobolev2011},
Theorems 6.1 and 6.2 in \cite{necasMethodesDirectesTheorie1967} and
\eqref{ineqc}, there are $v\in W^{1,\bar{t}}(\Omega)$, and a subsequence
$\{v_{n_{k}}\}$ of $\{v_{n}\}$ such that $\{\nabla v_{n_{k}}\}$
converges weakly to $\nabla v$ in $L^{\bar{t}}(\Omega)$ and
$\lim_{k\to\infty}\norm{v_{n_{k}}-v}_{L^{r}(\Omega)}=\lim_{k\to\infty}\norm{Bv_{n_{k}}-Bv}_{L^{\bar{t}}(\partial
  \Omega)}=0$. By Hahn-Banach's theorem,
$\norm{\nabla v}_{L^{\bar{t}}(\Omega)}\le \limsup_{k\to\infty}\norm{\nabla
  v_{n_{k}}}_{L^{\bar{t}}(\Omega)}=0$. Thus $\nabla v=0$ and
$\norm{v}_{L^{r}(\Omega)}=1$.  Since
$\lim_{k\to\infty}\norm{Bv_{n_{k}}-Bv}_{L^{\bar{t}}(\partial \Omega)}=0$, by
Theorem 4.9 in \cite[p.94]{brezisFunctionalAnalysisSobolev2011},
$Bv=0$ on $A$. By Theorem 3.1.3 in
\cite[p.64]{morreyMultipleIntegralsCalculus1966}, $v=0$, which
contradicts to $\norm{v}_{L^{r}(\Omega)}=1$. Thus we get \eqref{ineqb}
and \eqref{ineq}.
\end{itemize}
\end{proof}

 We have following norms on $C^{1}(\Omega,A)$ and $C^{1}(Q,A)$.
 \begin{lemma}
   \label{4l2}
   Suppose $A$ is admissible with respect to $\Omega$. Then
 \begin{itemize}[wide]
 \item[(i)] $\norm{\,\cdot\,}_{C,b}$ and $\norm{\,\cdot\,}_{C,B}$  are  norms of
   $C^{1}(\Omega,A)$.
 \item[(ii)] $\norm{\,\cdot\,}_{C,b,T}$ and $\norm{\,\cdot\,}_{C,B,T}$  are  norms of  $C^{1}(Q,A)$.
 \end{itemize}
  \end{lemma}
  \begin{proof}
    \begin{itemize}[wide]
    \item[$(i)$] Let $v$ be in $C^{1}(\Omega,A)\setminus \left\{0\right\}$. By \eqref{i1}, it sufficient to prove $\norm{v}_{b}>0$. We have
\begin{align}
    \label{a12}
    \begin{split}      0&\coex[v\not=0]{<}\left\{\int_{\Omega}|v|^{r}dx\right\}^{\frac{1}{r}}
      \\
      &\coex[\eqref{ineq}]{\le} C(r,\Omega,A)\left\{\int_{\Omega} \sum_{i=1}^{N} \abs{\frac{\partial v}{\partial x_i}}^{\overline{t}}dx\right\}^{\frac{1}{\overline{t}}} 
      \\
      &\coex[\text{H\"{o}lder}]{\le} C(r,\Omega,A)\left\{\int_{\Omega} b^{-\frac{\bar{t}}{2-\bar{t}}} dx\right\}^{\frac{2-\overline{t}}{2\overline{t}
      }}\left\{\int_{\Omega} \sum_{i=1}^{N} \abs{\frac{\partial v}{\partial x_i}}^{2}bdx\right\}^{\frac{1}{2}}.
    \end{split}
\end{align}
Thus we get $(i)$.
 
\item[$(ii)$] Let $v$ be in $C^{1}(Q,A)\setminus \left\{0\right\}$. We
  prove $\norm{v}_{b,T}>0$. We have
  \begin{align*} 0&\coex[v\not=0]{<}\int_{0}^{T}\left\{\int_{\Omega}|v(x,t)|^{r}dx\right\}^{\frac{1}{r}}dt\\
    &\coex[\eqref{a12}]{\le}
      C(r,\Omega,A)\norm{b^{-\frac{\bar{t}}{2-\bar{t}}}}_{L^{1}(\Omega)}^{\frac{2-\overline{t}}{2\overline{t}}}\int_{0}^{T}\left\{\int_{\Omega}
      |\nabla_{x} v|^{2}(x,t)bdx\right\}^{\frac{1}{2}}dt\\
                  &\coex[\text{H\"{o}lder}]{\le} C(r,\Omega,A)\norm{b^{-\frac{\bar{t}}{2-\bar{t}}}}_{L^{1}(\Omega)}^{\frac{2-\overline{t}}{2\overline{t}}}T^{\frac{1}{2}}\left\{\int_{0}^{T}\int_{\Omega} |\nabla_{x} v|^{2}(x,t)|bdxdt\right\}^{\frac{1}{2}},
  \end{align*}
Thus we get $(ii)$.
\end{itemize}
\end{proof}
\begin{definition}
  Assume \eqref{c1} and \eqref{c3} hold. We denote by
  \begin{center}
    \begin{tabular}[ht]{ccc}
      $\left(W_{b,A}(\Omega),\norm{\,\cdot\,}_{b} \right)$ &  &
                                                              $\left(C^{1}(\Omega,A),\norm{\,\cdot\,}_{C,b} \right)$\\
      $\left(W_{B,A}(\Omega),\norm{\,\cdot\,}_{B} \right)$ &&
                                                              $\left(C^{1}(\Omega,A),\norm{\,\cdot\,}_{C,B} \right)$\\
      $\left(W_{b,A,T}(Q),\norm{\,\cdot\,}_{b,T} \right)$ & $\qquad$ the
                                                             completion
      of $\qquad$ &
                                                             $\left(C^{1}(Q,A),\norm{\,\cdot\,}_{C,b,T} \right)$\\
      $\left(W_{B,A,T}(Q),\norm{\,\cdot\,}_{B,T} \right)$ &&
                                                             $\left(C^{1}(Q,A),\norm{\,\cdot\,}_{C,B,T} \right)$\\
      $\left(V_{B,A,T}(Q),\norm{\,\cdot\,}_{V,B,T} \right)$ && $\left(C^{1}(Q,A),\norm{\,\cdot\,}_{C,V,B,T} \right),$
    \end{tabular}
  \end{center}
  respectively.
   \label{d44}
 \end{definition}
By \eqref{i1}, we have
    \begin{align}
      \begin{split}
        &W_{B,A}(\Omega) \subset W_{b,A}(\Omega)\subset L^{r}(\Omega)\quad \text{
   and }\\
   &V_{B,A,T}(Q) \subset W_{B,A,T}(Q) \subset W_{b,A,T}(Q) \subset L^{r}(Q), 
      \end{split}
  \label{i2}
    \end{align}
If $A$ is admissible, then by the proof of Lemma \ref{4l2}, 
\begin{equation}
  \norm{w}_{b,T}\le \norm{w}_{B,T}\qquad\forall~w\in W_{B,A,T}(Q).
  \label{i2b}
\end{equation}
 Since we shall use $W_{b,A}(\Omega)$, $W_{B,A}(\Omega)$, $W_{B,A,T}(Q)$, and $V_{B,A,T}(Q)$  to study parabolic equations, we need relations between them with functions having generalized derivatives. We have the following results.
   \begin{lemma} Let $A$ be  admissible with respect to $\Omega$. If
     $w$ is in $W_{b,A}(\Omega)$  ($V_{B,A,T}(Q)$, respectively), then
     $w$ has the generalized derivative $\displaystyle \frac{\partial
       w}{\partial x_{i}}$ for every $i$ in
     $\left\{1,\cdots,N\right\}$ ($\left\{1,\cdots,N+1\right\}$,
     respectively). Moreover, the followings hold:
     \begin{align}
       \norm{v}_{b}
       &= \left\{\int_{\Omega}|\nabla_{x}v|^{2}bdx\right\}^{\frac{1}{2}}\qquad\forall~v\in W_{b,A}(\Omega),   
     \label{i3e}\\
     \norm{v}_{B}
     &\ge \left\{\int_{\Omega}\sum_{i,j=1}^{N}\frac{\partial v}{\partial x_{i}}\frac{\partial v}{\partial x_{j}}b_{ij}dx\right\}^{\frac{1}{2}} \qquad\forall~v\in W_{B,A}(\Omega),
     \label{i3eb}\\
     \norm{v}_{b,T}
     &= \left\{\int_{0}^{T}\int_{\Omega}|\nabla_{x}v|^{2}(x,t)b(x)dxdt\right\}^{\frac{1}{2}}\qquad\forall~v\in W_{b,A,T}(Q),   
   \label{i3}\\
   \norm{v}_{B,T}
   &\ge \left\{\int_{0}^{T}\int_{\Omega}\sum_{i,j=1}^{N}\frac{\partial v}{\partial x_{i}}(x,t)\frac{\partial v}{\partial x_{j}}(x,t)b_{ij}(x)dxdt\right\}^{\frac{1}{2}} \qquad\forall~v\in W_{B,A,T}(Q),
   \label{i3b}\\
   \norm{v}_{V,B,T}
   &\ge \norm{\frac{ \partial v}{\partial t}}_{L^{2}(Q)} +\left\{\int_{Q}\sum_{i,j=1}^{N}\frac{\partial v}{\partial x_{i}}\frac{\partial v}{\partial x_{j}}b_{ij}dz\right\}^{\frac{1}{2}} \qquad\forall~v\in V_{B,A,T}(Q).
\label{i3c}
\end{align}
\label{5z}
  \end{lemma}
  \begin{proof}
We denote $L_{b}^{2}(Q)$ the norm space of all measurable function $w$ on $Q$ such that
  \[\norm{w}_{b}\coloneqq \left\{\int_{Q}|w|^{2}bdz\right\}^{\frac{1}{2}}<\infty.\]  
  Since $2>\bar{t} >\frac{2N^{2}+2N-2}{N^{2}+2N-1}\ge 1$, we have $\frac{\bar{t}}{2-\bar{t}}\ge 1$ and
  \begin{equation} L^{\frac{\bar{t}}{2-\bar{t}}}(Q) \subset L^{1}(Q).  
   \label{a13}
\end{equation}
Let $\left\{u_{m}\right\}$ be in $C^{1}(Q,A)$ such that
$\left\{u_{m}\right\}$ is a Cauchy sequence in
$\left(C^{1}(Q,A),\norm{\,\cdot\,}_{B,T} \right)$.  By \eqref{i2b} and
\eqref{a12}, $\left\{u_{m}\right\}$ is a Cauchy sequence in $L^{r}(Q)$
and
$\left\{\displaystyle\frac{\partial u_{m}}{\partial x_{j}}\right\}$ a
Cauchy sequence in $L^{2}_{b}(Q)$ for every
$j\in\left\{1,\cdots,N\right\}$. Because $L^{r}(Q)$ and $L^{2}_{b}(Q)$
are Banach spaces, $\left\{u_{m}\right\}$ converges to $u$ in
$L^{r}(Q)$ and
$\displaystyle\left\{\frac{\partial u_{m}}{\partial x_{j}}\right\}$
converges to $v_{j}$ in $L^{2}_{b}(Q)$ for every $j$ in
$\left\{1,\cdots,N\right\}$. Thus, we have
\begin{align*}
  \lim_{m\to \infty}\int_{Q}|u_{m}-u|dz
    &\coex[r\ge 1] = 0,
  \intertext{and}
  \lim_{m\to \infty}\int_{Q}\abs{\frac{\partial u_{m}}{\partial
  x_{i}}-v_{j}}dz
  &\coex[\text{H\"{o}lder}]{\le} \lim_{m\to
    \infty}\left\{\int_{Q}b^{-1}dz\right\}^{\frac{1}{2}}
    \left\{\int_{Q}\abs{\frac{\partial u_{m}}{\partial
    x_{i}}-v_{j}}^{2}b dz\right\}^{\frac{1}{2}}\\
  &\coex[\eqref{c3} \text{ and } \eqref{a13}]{=} 0\qquad\forall~j=1,\cdots,N.
\end{align*}
Hence,
\begin{align*}
  \int_{Q}v_{i}\varphi dz
  &= \lim_{n\to\infty} \int_{Q}\frac{\partial u_{n}}{\partial
    x_{i}}\varphi dz=-\lim_{n\to\infty} \int_{Q}u_{n}\frac{\partial
    \varphi}{\partial x_{i}} dz\\
  &= -\int_{Q}u\frac{\partial \varphi}{\partial x_{i}} dz\qquad\forall~\varphi\in C^{1}_{c}(Q).
\end{align*}

Therefore, $v_{i}$ is the generalized partial derivative of $u$ along
direction $x_{i}$. Similary $v_{0}$ is the generalized partial
derivative of $u$ along direction $t$. Hence $u$ has  first order
generalized partial derivatives. On other hand, 
\begin{align*}
  \norm{u}_{b,T}
  &=\lim_{n\to\infty}\left\{\int_{Q}\sum_{i=1}^{N}\abs{\frac{\partial
  u_{m}}{\partial x_{i}}}^{2}b dz\right\}^{\frac{1}{2}}
  =\left\{\int_{Q}\abs{( v_{1},\cdots,v_{N})}^{2}b
    dz\right\}^{\frac{1}{2}}\\
  &=\left\{\int_{Q}\sum_{i=1}^{N}\abs{\frac{\partial u}{\partial x_{i}}}^{2}bdz\right\}^{\frac{1}{2}}. 
\end{align*}
Thus we get \eqref{i3}.

By Theorem 4.9 in \cite[p.94]{brezisFunctionalAnalysisSobolev2011}, we
can suppose  $\displaystyle\left\{\frac{\partial u_{m}}{\partial
    x_{j}}\right\}$  converges  to  $\displaystyle \frac{\partial
  u}{\partial x_{j}}$ a.e. on $Q$ for every $j$ in
$\left\{1,\cdots,N\right\}$, respectively. Therefore,
\begin{align*}
&\left\{\int_{0}^{T}\int_{\Omega} \sum_{i,j=1}^{N}\frac{\partial
  u}{\partial x_{i}}(x,t)\frac{\partial u}{\partial
  x_{j}}(x,t)b_{ij}(x)dxdt\right\}^{\frac{1}{2}}\\
&\coex[\text{Fatou}]{\le}
  \liminf_{m\to\infty}\left\{\int_{0}^{T}\int_{\Omega}\sum_{i,j=1}^{N}\frac{\partial
  u_{m}}{\partial x_{i}}(x,t)\frac{\partial u_{m}}{\partial
  x_{j}}(x,t)b_{ij}(x)dxdt\right\}^{\frac{1}{2}}\\
  &\coex{=}\lim_{m\to\infty}\norm{u_{m}}_{B,T}=\norm{u}_{B,T}. 
\end{align*}
Thus we get \eqref{i3b}.

By a similar argument, with $Q$ is replaced by $\Omega$, we get
\eqref{i3e} and \eqref{i3eb} for $v \in W_{b,A}(\Omega)$ and $v \in
W_{B,A}(\Omega)$, respectively.

Similarly, let $\left\{u_{n}\right\}$  be a
 Cauchy sequence in $\left(C^{1}(Q,A),\norm{\,\cdot\,}_{V,B,T} \right)$, then
 there are $u\in V_{B,A,T}(Q)$ and  $v_{N+1}\in L^{r}(Q)\subset
 L^{2}(Q)$ such that $\left\{u_{n}\right\}$  converges to  $u$ in
 $V_{B,A,T}(Q)$ and in $L^{2}(Q)$, and
 $\displaystyle\left\{\frac{\partial u_{n}}{\partial t} \right\}$
 converges to $v_{N+1}$ in $L^{r}(Q)\subset L^{2}(Q) $. Arguing as
 above,  we obtain $\displaystyle\frac{\partial u}{\partial
   t}=v_{N+1}$ and \eqref{i3c}.
 \end{proof} 
  \begin{lemma} Let $A$ be  admissible with respect to $\Omega$. Assume \eqref{c1} and \eqref{c3} hold. Then 
 \begin{equation}\norm{w}_{L^{r}(Q)}\le C(r,Q,A)\norm{w}_{b,T}\le C(r,Q,A)\norm{w}_{B,T}\qquad\forall~ w\in W_{B,A,T}(Q).
\label{cb}
\end{equation}
  \label{4l2b}
  \end{lemma}
  \begin{proof}
 Let $u$ be in $W_{b,A,T}(Q)$ and $\left\{u_{n}\right\}_{n}$ be a sequence in $C^{1}(Q,A)$ such that $\left\{u_{n}\right\}_{n}$ converges to $u$ in $W_{b ,A,T}(Q)$. By \eqref{ineq}, $\left\{u_{n}\right\}_{n}$ converges to $u$ in $L^{r}(Q)$. Thus
 \begin{align*}
 \norm{u}_{L^{r}(Q)}= \lim_{n\to \infty}\norm{u_{n}}_{L^{r}(Q)}
   &\coex[\eqref{a12}]{\le} C(r,Q,A)\lim_{n\to \infty}\left\{\int_{Q}
     \sum_{i=1}^{N} \abs{\frac{\partial u_{n}}{\partial x_i}
     (x,t)}^{2}bdz\right\}^{\frac{1}{2}}\\
   &\coex{\le}  C(r,Q,A)\left\{\int_{Q} \sum_{i=1}^{N}
     \abs{\frac{\partial u}{\partial x_i}
     (x,t)}^{2}bdz\right\}^{\frac{1}{2}}\\
   &\coex{=} C(r,Q,A)\norm{u}_{b,T}\\
   &\coex[\eqref{i2b}]{\le} C(r,Q,A)\norm{u}_{B,T}.
 \end{align*}
and we obtain the lemma.
 \end{proof} 
 \section{Superpositions  in $W_{B,A,T}(Q)$ and $V_{B,A,T}(Q)$}\label{app}
We need  following results to study the regularity of solutions of parabolic equations.
\begin{lemma}
  Let $A$ be admissible respect to $\Omega$,
  $\alpha_0 \in [-\infty,\infty)$, $\phi\in C^{1}((\alpha_0,\infty))$, $v$
  be in $W_{B,A,T}(Q)$, $\left\{v_{n}\right\}$ be a sequence in
  $C^{1}(Q,A)$ such that
  $\bigcup_{n=1}^{\infty} v_{n}(Q)\cup v(Q) \subset
  (\alpha_0,\infty)$. We assume \eqref{c1}, \eqref{c3}, and the followings hold,
  \begin{enumerate}
  \item[$(a)$]
  $\phi\in C^{1}((\alpha_0,\infty))$, $\phi'$ is  uniformly continuous on $(\alpha_0,\infty)$ and 
  $\norm{\phi'}_{L^{\infty}(\RR)}\le M$  with  a positive real number
  M;
\item[$(b)$]
  $\lim_{n\to\infty}\norm{v_{n}-v}_{B,T}=0$.
  \end{enumerate}
  Then 
  \begin{itemize}
  \item[$(i)$] $\left\{v_{n}\right\}$ has a subsequence
  $\left\{v_{n_{k}}\right\}$ such that
  \begin{equation}
  \lim_{k\to\infty}\norm{\phi\circ v_{n_{k}}-\phi\circ v}_{B,T}=0.
   \label{2l51}
   \end{equation}
 \item[$(ii)$] If $\phi(0)=0$, then $\phi\circ v \in W_{B,A,T}(Q)$.
 \item[$(iii)$] If  $v\in V_{B,A,T}(Q)$, $\phi(0)=0$ and $\displaystyle\lim_{n\to\infty}\norm{\frac{\partial v_{n}}{\partial t} -\frac{\partial v}{\partial t}}_{L^{2}(Q)}=0$, then $\phi\circ v \in V_{B,A,T}(Q)$.
  \end{itemize} 
     \label{2w1ba}
\end{lemma}
\begin{proof}  We  prove the lemma in four steps.
  \begin{description}[wide]
\item[Step 1.]
 Since $\left\{v_{n}\right\}$
converges in $W_{B,A,T}(Q)$, there is a positive real number $M_{1}$
such that
\begin{equation}
  \norm{v_{n}}_{B,T}\le M_{1}\qquad \forall~n\in\NN.
  \label{47z0}
\end{equation}
We have
\[
  \lim_{n\to \infty}\int_{Q}|v_n{}-v|dz \coex[r > 1 \text{ and }\eqref{cb}]{=} 0.
\]
By Theorem 4.9 in \cite[p.94]{brezisFunctionalAnalysisSobolev2011},
there is a subsequence $\left\{v_{n_{k}}\right\}$ of
$\left\{v_{n}\right\}$ such that $\left\{v_{n_{k}}\right\}$ converges
to $v$ a.e. on $Q$.

Let $\varepsilon \in(0,\infty)$ and $\delta \in (0,\varepsilon)$. Since $Q$
is bounded, by Egorov's Theorem in
\cite[p.115]{brezisFunctionalAnalysisSobolev2011}, $Q$ has a subset
$A_{\delta}$ such that $\abs{Q\setminus A_{\delta}}\le \delta$ and
$\left\{v_{n_{k}}\right\}$ converges uniformly to $v$ on
$A_{\delta}$. Because $\phi'$ is uniformly continuous on $\RR$, there
is an integer $k_{\varepsilon}$ such that 
\begin{align*}
  \abs{\phi'(v(z))-\phi'(v_{n_{k}}(z))}
  &\le \varepsilon \qquad\forall~z \in A_{\delta}, k\ge k_{\varepsilon},\\
  \norm{v - v_{n_k}}_{B,T}
  &\le \varepsilon\qquad\forall~ k\ge k_{\varepsilon}.
\end{align*}
This implies
 \begin{align}
   \abs{\phi'(v_{n_{k'}}(z))-\phi'(v_{n_{k}}(z))}
   &\le 2\varepsilon \qquad\forall~z \in A_{\delta}, k', k\ge k_{\varepsilon},
  \label{47z2}\\
   \norm{v_{n_{k'}} -  v_{n_{k}}}_{B,T}
   &\le 2\varepsilon\qquad\forall~ k',k\ge k_{\varepsilon}.
  \label{247z3}
 \end{align}
 
Define
\begin{equation}
  \vuong w \coloneqq  \sum_{i,j=1}^{N}\frac{\partial w}{\partial x_{i}}\frac{\partial w}{\partial x_{j}}b_{ij}\qquad\forall~w\in W_{B,A,T}(Q).
 \label{vuong}
 \end{equation}
 By \eqref{c1} and \eqref{i3eb}, $\vuong v$ and $\vuong v_{n_{k_{\varepsilon}}}$ are non-negative and integrable on $Q$. Therefore, we can choose  $\delta$ such that
 \begin{align}
   &\int_{Q\setminus A_{\delta}}\vuong v dx \le \varepsilon^{2}.
  \label{247z4b}\\
&\int_{Q\setminus A_{\delta}}\vuong v_{n_{k_{\varepsilon}}}dz\le \varepsilon^{2}.
  \label{247z4}
\end{align}
Thus
\begin{align*}
\left\{\int_{Q\setminus A_{\delta}} \vuong
  v_{n_{k}}dz\right\}^{\frac{1}{2}}
  &\coex[\text{Minkowski}]{\le}
  \left\{\int_{Q\setminus A_{\delta}}\vuong
    v_{n_{k_{\varepsilon}}}dz\right\}^{\frac{1}{2}}+\left\{\int_{Q\setminus
    A_{\delta}}\vuong (v_{n_{k_{\varepsilon}}}-
    v_{n_{k}})dz\right\}^{\frac{1}{2}}\\
  &\coex[\eqref{247z4} \text{ and }\eqref{247z3}]{\le}
    3\varepsilon\qquad\forall~ k\ge k_{\varepsilon},
\end{align*}
 or
 \begin{equation}
 \int_{Q\setminus A_{\delta}} \vuong v_{n_{k}}dz\le 9\varepsilon^{2} \qquad\forall~ k\ge k_{\varepsilon}. 
  \label{247z3b}
\end{equation}

Fix $x$ in $\Omega$, we define
\[
  \left\langle y,u \right\rangle= \sum_{i,j=1}^{N}y_{i}u_{j}b_{ij}(x)\qquad\forall~y=\left(y_{1},\cdots,y_{N} \right),~
  u=\left(u_{0},u_{1},\cdots,u_{N} \right)\in \RR^{N}.
\]
Then $\left\langle \cdot,\cdot \right\rangle$ is a
scalar product on $\RR^{N}$. Therefore,
\[
  \left\langle y+u,y+u \right\rangle
  \le \left(\left\langle y,y \right\rangle^{\frac{1}{2}} + \left\langle u,u \right\rangle^{\frac{1}{2}}  \right)^{2}\le 2(\left\langle y,y \right\rangle+\left\langle u,u \right\rangle),
\]
for all $y=(y_{1},\cdots,y_{N}),~ u=(u_{1},\cdots,u_{N})\in \RR^{N}.$

Put  $\displaystyle h_{i}\coloneqq
\left[\phi'(v_{n_{k'}})-\phi'(v_{n_{k}}) \right]\frac{\partial
  v_{n_{k'}}}{\partial x_{i}} + \phi'(v_{n_{k}})\frac{\partial
  (v_{n_{k'}}-v_{n_{k}})}{\partial x_{i}} \eqqcolon y_i + u_i$ for every $i$ in $\left\{1,\cdots,N\right\}$ and $h=(h_{1},\cdots,h_{N})$. We have    
\begin{align}
  \label{b1}
  \begin{split}
    &\norm{\phi\circ v_{n_{k'}} - \phi\circ v_{n_{k}}}_{B,T} ^{2}\\
    &\coex{=} \int_{Q}\vuong \left(\phi\circ v_{n_{k'}} - \phi\circ
      v_{n_{k}} \right)dz\\
    &\coex{=} \int_{Q} \sum_{i,j=1}^{N} \left[\phi'(v_{n_{k'}})\frac{\partial  v_{n_{k'}}
      }{\partial x_{i}} - \phi'(v_{n_{k}})\frac{\partial
      v_{n_{k}}}{\partial x_{i}} \right] \left[\phi'(v_{n_{k'}})\frac{\partial
      v_{n_{k'}} }{\partial x_{j}} - \phi'(v_{n_{k}})\frac{\partial
      v_{n_{k}}}{\partial x_{j}} \right]b_{ij} dz\\
    &\coex{=}\int_{Q} \left\langle h,h \right\rangle dz\\
    &\coex{\le}
      2\left\{\int_{Q}\sum_{i,j=1}^{N}\left[(\phi'(v_{n_{k'}})-\phi'(v_{n_{k}})\right]^{2}\frac{\partial
      v_{n_{k'}} }{\partial x_{i}}\frac{\partial  v_{n_{k'}}
      }{\partial x_{j}}b_{ij} dz \right.\\
    &\coex{~} \qquad + \left.\int_{Q} \sum_{i,j=1}^{N}\phi'(v_{n_{k}})^{2}\frac{\partial
      (v_{n_{k'}}-v_{n_{k}}) }{\partial x_{i}} \frac{\partial
      (v_{n_{k'}}-v_{n_{k}}) }{\partial x_{j}}b_{ij} dz\right\}\\
    &\coex[\eqref{247z3} \text{ and } (a)]{\le} 
 2\int_{A_{\delta}}\sum_{i,j=1}^{N}\left[\phi'(v_{n_{k'}})-\phi'(v_{n_{k}})\right]^{2}\frac{\partial
      v_{n_{k'}} }{\partial x_{i}}\frac{\partial  v_{n_{k'}}
      }{\partial x_{j}}b_{ij} dz\\
    &\coex{~} \qquad
      +2\int_{Q\setminus A_{\delta}} \sum_{i,j=1}^{N}\left[\phi'(v_{n_{k'}})-\phi'(v_{n_{k}}) \right]^{2}\frac{\partial  v_{n_{k'}}
      }{\partial x_{i}}\frac{\partial  v_{n_{k'}} }{\partial
      x_{j}}b_{ij} dz+ 4M^{2}\varepsilon^{2}\\
    &\coex[\eqref{47z2}]{\le} 2\varepsilon^{2}\int_{Q}\vuong
      v_{n_{k'}}dz+8M^{2}\int_{Q\setminus A_{\delta}}\vuong
      v_{n_{k'}} dz + 4M^{2}\varepsilon^{2}\\
    &\coex[\eqref{47z0} \text{ and }\eqref{247z4}]{\le}
      \varepsilon^{2}\left(2 M_{1}^{2}+ 72 M^{2}+ 4 M^{2}\right)\qquad\forall~k',k\ge k_{\varepsilon}.
  \end{split}
\end{align}   
 Therefore $\left\{\phi \circ v_{n_{k}}\right\}_{k}$ is a Cauchy
 sequence in $W_{B,A,T}(Q)$, thus it converges to a $w$ in  $W_{B,A,T}(Q)$.

\item[Step 2.]  By \eqref{i2b}, $\left\{\phi \circ
    v_{n_{k}}\right\}_{k}$ converges to $w$ in  $W_{b,A,T}(Q)$. We
  shall prove $w= \phi\circ v$. It is sufficient to  prove that
  $\left\{v_{n_{k}}\right\}_{k}$ has a subsequence
  $\left\{v_{n_{k_{l}}}\right\}_{l}$ such that $\left\{\phi \circ
    v_{n_{k_{l}}}\right\}_{l}$ converges to $\phi\circ v$ in
  $W_{b,A,T}(Q)$.
  
   Let $\varepsilon \in(0,\infty)$. By \eqref{i2b} and Lemma \ref{5z},  $|\nabla v|^{2}b$  is in $L^{1}(Q)$. Then  there is a positive real number $\delta$ such that
  \begin{equation}
  \int_{Q\setminus E}|\nabla v|^{2}bdz<\varepsilon\qquad\forall~E\subset Q~ \text{
    with } ~\abs{Q\setminus E}\le \delta.
  \label{47z1}
\end{equation}

By $(b)$ and Lemma \ref{4l2b}, we have
\[\lim_{k\to \infty}\int_{Q}\abs{v_{n_{k}}-v}dz
  \coex[1<r \text{ and }\eqref{cb}]{\le} 0.
\]
By Theorem 4.9 in \cite[p.94]{brezisFunctionalAnalysisSobolev2011},
there is a subsequence $\left\{v_{n_{k_{l}}}\right\}$ of
$\left\{v_{n_{k}}\right\}$ such that $\left\{v_{n_{k_{l}}}\right\}$
converges  to $v$ a.e. on $Q$.

Since $Q$ is bounded, by Egorov Theorem in
\cite[p.115]{brezisFunctionalAnalysisSobolev2011}, $Q$ has a subset
$E_{\delta}$ such that $|Q\setminus E_{\delta}|\le \delta$ and
$\left\{v_{n_{k}}\right\}$ converges uniformly to $v$ on
$E_{\delta}$. Because $\phi'$ is uniformly continuous on $\RR$, there
is an integer $l_{\varepsilon}$ such that
 \begin{align}
   \abs{\phi'(v(x))-\phi'\left(v_{n_{k_{l}}}(x) \right)}
   &\coex{\le} \varepsilon \qquad\forall~x \in E_{\delta},~ l\ge l_{\varepsilon},
   \label{47z2c}
     \intertext{and}
     \norm{v -   v_{n_{k_{l}}}}_{b,T}
   &\coex[\eqref{i2b} \text{ and } (b)]{\le} \varepsilon\qquad\forall~ l\ge l_{\varepsilon}.
  \label{47z3c}
 \end{align}
Thus     
\begin{align*}
  \norm{\phi\circ v - \phi\circ v_{n_{k_{l}}}}_{b,T }
  &\coex[\eqref{i3}]{=} \left\{\int_{Q}\sum_{i=1}^{N}\abs{\phi'(v)\frac{\partial
    v}{\partial x_{i}}- \phi'\left(v_{n_{k_{l}}} \right)\frac{\partial
    v_{n_{k_{l}}}}{\partial x_{i}}}^{2}bdz\right\}^{\frac{1}{2}}\\
  &\coex[\text{Minkowski}]{\le} \left\{\int_{Q} \left[\phi'(v)-\phi'\left(v_{n_{k_{l}}} \right)
    \right]^{2}\sum_{i=1}^{N} \left(\frac{\partial
    v}{\partial x_{i}} \right)^{2}bdz\right\}^{\frac{1}{2}}\\
  &\coex{~} \qquad+\left\{\int_{Q}\abs{\phi'\left(v_{n_{k_{l}}}
    \right)}^{2}\sum_{i=1}^{N} \left(\frac{\partial
    v}{ \partial x_{i}}-\frac{\partial v_{n_{k_{l}}}}{\partial
    x_{i}} \right)^{2}bdz\right\}^{\frac{1}{2}} \\
  &\coex[\text{(a)}]{\le} \left\{\int_Q \left( \chi_{E_{\delta}} +
    \chi_{Q\setminus E_{\delta}} \right) \left[\phi'(v)-\phi'\left(v_{n_{k_{l}}} \right)
    \right]^{2}\abs{\nabla_x v}^2bdz\right\}^{\frac{1}{2}}\\
  &\coex{~} \qquad+M \left\{\int_Q\sum_{i=1}^{N} \left(\frac{\partial
    v}{ \partial x_{i}}-\frac{\partial v_{n_{k_{l}}}}{\partial
    x_{i}} \right)^{2}bdz\right\}^{\frac{1}{2}} \\
    &\coex[\substack{\text{Minkowski}\\\eqref{47z3c}}]{\le} \left\{\int_Q  \chi_{E_{\delta}} \left[\phi'(v)-\phi'\left(v_{n_{k_{l}}} \right)
    \right]^{2}\abs{\nabla_x v}^2bdz\right\}^{\frac{1}{2}}\\
  &\coex{~} \qquad+ \left\{\int_Q 
    \chi_{Q\setminus E_{\delta}} \left[\phi'(v)-\phi'\left(v_{n_{k_{l}}} \right)
    \right]^{2}\abs{\nabla_x v}^2bdz\right\}^{\frac{1}{2}}\\
  &\coex{~} \qquad + M\varepsilon\\
  &\coex[\eqref{47z2c} \text{ and } (b)]{\le} \varepsilon \left\{\int_Q  \chi_{E_{\delta}} \abs{\nabla_x v}^2bdz\right\}^{\frac{1}{2}}\\
  &\coex{~} \qquad+ 2M \left\{\int_Q 
    \chi_{Q\setminus E_{\delta}}\abs{\nabla_x v}^2bdz\right\}^{\frac{1}{2}}\\
  &\coex{~} \qquad + M\varepsilon\\
&\coex[\eqref{47z1}]{\le}
  \varepsilon\norm{v}_{b,T}+2M\varepsilon^{2}+M\varepsilon\qquad\forall~l \ge l_{\varepsilon},
\end{align*}
which implies $\left\{\phi\circ v_{n_{k_{l}}}\right\}_{l}$ converges
to $\phi\circ v$ in $W_{b,A,T}(Q)$.

 Since $\left\{\phi \circ v_{n_{k_{l}}}\right\}_{l}$ converges to $w$
 in $W_{b,A,T}(Q)$, $w= \phi\circ v$. By Step 1, we get \eqref{2l51}.
\item[Step 3.] When  $\phi(0)=0
 $,  we see that $\left\{\phi\circ v_{n_{k}}\right\}$ is contained in
 $C^{1}(Q,A)$. Therefore $\phi\circ v$ is in $W_{B,A,T}(Q)$.
\item[Step 4.] Now we  prove $(iii)$. We have 
\begin{align*}
\int_{Q}\abs{\frac{\partial \left[\phi\circ v_{n}- \phi\circ v \right]}{\partial
  t}}^{2}dz
  &= \int_{Q}\abs{\phi'(v_{n})\frac{\partial v_{n}}{\partial
    t}- \phi'(v)\frac{\partial v}{\partial t}}^{2}dz\\
  &= \int_{Q}\abs{\left(\phi'(v_{n})-\phi'(v) \right)\frac{\partial v_{n}}{\partial t}+
  \phi'(v) \left(\frac{\partial v_{n}}{\partial t}-\frac{\partial v}{\partial
    t} \right)}^{2}dz.
\end{align*}

Arguing as in above steps, we can suppose
$\left\{\displaystyle\frac{\partial\left(\phi\circ v_{n_{k_l}}
    \right)}{\partial t}\right\}_{l}$ converges to
$\displaystyle\frac{\partial (\phi\circ v)}{\partial
  t} $ in $L^{2}(Q)$. It follows that $\phi\circ v$ is in
$V_{B,A,T}(Q)$.
\end{description}
\end{proof}

\begin{lemma}
  Let $A$ be admissible with respect to $\Omega$, $v$ be in
  $W_{B,A,T}(Q)$ (respectively, $V_{B,A,T}(Q)$),
  $v^{+}\coloneqq \max\left\{0,v\right\}$,
  $v^{-}\coloneqq \max\left\{0,-v\right\}$.  Assume \eqref{eq:ccc}
  holds. Then 
  \begin{itemize}
  \item[ $(i)$]  $v^{+}$, $v^{-}$ and $|v|$ are in
  $W_{B,A,T}(Q)$ (respectively, $V_{B,A,T}(Q)$).
\item[$(ii)$] If $v$ is
  non-negative on $\Omega$, there is a sequence of non-negative
  functions $\left\{w_{n}\right\}$ in $C^{1}(Q,A)$ such that
  $\lim_{n\to\infty}\norm{w_{n}-v}_{B,T}=0$ (respectively,
  $\lim_{n\to\infty}\norm{w_{n}-v}_{V,B,T}=0$).
\item[$(iii)$] If $\eta\in C^{1}_{c}(\RR^{N+1})$, then $\eta v$ is in
  $W_{B,A,T}(Q)$ (respectively $V_{B,A,T}(Q)$).
  
  \end{itemize}
 \label{w1bb}
\end{lemma}
\begin{proof} We only prove for  $W_{B,A,T}(Q)$. Using the techniques
  of these proofs and $(iii)$ in  Lemma \ref{2w1ba}, we obtain the
  results for $V_{B,A,T}(Q)$.
  \begin{itemize}[wide]
  \item[$(i)$] Let $w_{1}$ be $v^{+}$, $w_{2}$ be $v^{-}$, $w_{3}$ be
    $|v|$, and $\left\{v_{n}\right\}$ be a sequence in $C^{1}(Q,A)$
    such that $\left\{v_{n}\right\}$ converges to $v$ in
    $W_{B,A,T}(Q)$.
    
Let $\beta\in (1,2)$. Define
 \begin{align}
   \Phi_{\beta} (\xi)
   &= \begin{cases}
 0 &\text{ if }\xi \in (-\infty,0],  \\
  \xi^{\beta}&\text{ if }\xi \in (0,1),  \\
     \beta(\xi-1)+1  &\text{ if } \xi \in [1 ,\infty).
\end{cases}
\label{z7}
\intertext{We have}
     \Phi_{\beta}' (\xi)
     &= \begin{cases}
 0 &\text{ if } \xi \in (-\infty,0],  \\
 \beta \xi^{\beta-1} & \text{ if } \xi \in (0,1),  \\
    \beta  & \text{ if } \xi \in [1 ,\infty).
\end{cases} 
 \label{z780bz} 
  \end{align}
Thus $\Phi_{\beta}$ and $\Phi_{\beta}'$ are monotone increasing on
$\RR$, $\Phi_{\beta}$ is in $C^{1}(\RR)$, $\Phi_{\beta}'$ is in
$L^{\infty}(\RR)$ and uniformly continuous on $\RR$. By Lemma
\ref{2w1ba},  $\Phi_{\beta}\circ v$ is in $W_{B,A,T}(Q)$ for every
$\beta$ in $(1,2)$.

By \eqref{z780bz}, we get
\begin{align}
&\Phi_{\beta} (\xi) \le \Phi_{\beta} (\zeta)\qquad \text{ if } 0\le \xi\le \zeta<\infty,
\label{hp0d}
\intertext{and}
&0\le \Phi_{\beta}' (\xi) \le \beta \quad \text{ and }\quad \lim_{\beta\to 1}\Phi_{\beta}' (\xi)=1\qquad \forall~ \xi\in (0,\infty).
\label{hp0db}
\end{align}

  Let $y$ be $t$ or $x_{j}$ where $j\in
  \left\{1,\cdots,N\right\}$. Since $v$ has first order generalized
  partial derivatives by Lemma \ref{5z} and  \cite[Lemma 7.6]{gilbargEllipticPartialDifferential2001}, we have
\begin{align}
  \frac{\partial v^{+}}{\partial y}(x)
  &= \begin{cases}
    0 &\text{ if }  v(x) \le 0,   \\
    \dfrac{\partial v}{\partial y}(x) &\text{ if } v(x) > 0, 
\end{cases}
\label{5z1}\\
  \frac{\partial v^{-}}{\partial y}(x)
  &= \begin{cases}
 0 & \text{ if } v(x) \ge 0,  \\
       -\dfrac{\partial v}{\partial y}(x) &\text{ if } v(x) < 0, 
\end{cases}
\label{5z2}\\
  \frac{\partial |v|}{\partial y}(x)
  &= \begin{cases}
  \dfrac{\partial v}{\partial y}(x) &\text{ if } v(x) > 0,  \\
 0 &\text{ if } v(x)= 0,  \\
       -\dfrac{\partial v}{\partial y}(x) &\text{ if } v(x) < 0. 
\end{cases}
\label{5z3}
\end{align}
Put $Q^{+}=\{x\in Q : v(x)>0\}$.
By \eqref{z7}, $\Phi_{1+\frac{1}{n}}\circ v=\Phi_{1+\frac{1}{n}}\circ v^{+}$. We have
\begin{align}
  \label{bzz}
  \int_{Q^{+}}\sum_{i,j=1}^{N}\frac{\partial v^{+}}{\partial
  x_{i}}\frac{\partial v^{+}}{\partial x_{j}}b_{ij}dz
  \coex[\eqref{5z1}]{\le}
  \int_{Q}\frac{\partial v}{\partial x_{i}}\frac{\partial v}{\partial
  x_{j}}b_{ij}dz
  \coex[\eqref{i3b}]{<} \infty.
\end{align}
Moreover,
\begin{align*}
&\lim_{k\to \infty}\int_{Q}\sum_{i,j=1}^{N} \left(\frac{\partial
  \left(\Phi_{1+\frac{1}{k}}\circ v \right)}{\partial x_{i}}-\frac{\partial
  v^{+}}{\partial x_{i}} \right) \left(\frac{\partial
  \left(\Phi_{1+\frac{1}{k}}\circ v \right)}{\partial x_{j}}-\frac{\partial
  v^{+}}{\partial x_{j}} \right)b_{ij} dz \\
&\coex[\eqref{5z1}]{=}
  \lim_{k\to \infty }\int_{Q^{+}}\sum_{i,j=1}^{N}
  \left[\Phi_{1+\frac{1}{k}}'\left(v(x) \right)-1 \right]^{2}
  \left(\frac{\partial v^{+}}{\partial x_{i}}\frac{\partial
  v^{+}}{\partial x_{j}}b_{ij} \right) dz\\
  &\coex[\substack{\eqref{hp0db},\eqref{bzz}\\ \text{Lebesgue DCT}}]{=}0.
\end{align*}

It implies 
\begin{align*}
&\lim_{m,n\to\infty}\norm{\Phi_{1+\frac{1}{n}}\circ v -
  \Phi_{1+\frac{1}{m}}\circ v}_{B,T}\\
  &\coex[\eqref{vuong},~\text{Minkowski}]{\le} \lim_{n\to\infty}\left\{\int_{Q}\vuong
    \left[\left(\Phi_{1+\frac{1}{n}}\circ v \right)-v^{+} \right]
    dz\right\}^{\frac{1}{2}}\\
  &\coex{~} \qquad+\lim_{m\to\infty}\left\{\int_{Q}\vuong
    \left[\left(\Phi_{1+\frac{1}{m}}\circ v \right)-v^{+} \right]
    dz\right\}^{\frac{1}{2}}\\
  &\coex{=}0.
\end{align*}
 Thus, $\left\{\Phi_{1+\frac{1}{n}}\circ v\right\}$ is a Cauchy
 sequence in $W_{B,A,T}(Q)$ and converges to some $w$ in
 $W_{B,A,T}(Q)$. We shall prove $w=v^{+}$.  By Lemma \ref{4l2b},
 $\left\{\Phi_{1+\frac{1}{n}}\circ v\right\}$  converges to $w$ in
 $W_{b,A,T}(Q)$.
 
 We have 
 \begin{align*}
   \left\{\int_{Q^{+}}\abs{\nabla v}^{2}bdz\right\}^{\frac{1}{2}}
     &\coex[\eqref{c1},\eqref{i3},\eqref{bzz}]{<}\infty,
   \intertext{and}
   \lim_{n\to\infty}\norm{\Phi_{1+\frac{1}{n}}\circ v-v^{+}}_{b,T}
   &\coex{=}\lim_{n\to\infty}\left\{\int_{\Omega}\abs{\nabla
     \left(\Phi_{1+\frac{1}{n}}\circ v- v^{+} \right)}^{2}bdz
     \right\}^{\frac{1}{2}}\\
   &\coex[\eqref{5z1}]{=}
     \lim_{n\to \infty}\left\{
     \int_{Q^{+}}[\Phi_{1+\frac{1}{n}}'(v(x))-1]^{2}|\nabla
     v|^{2}bdz\right\}^{\frac{1}{2}}\\
   &\coex[\substack{\eqref{hp0db} \\ \text{Lebesgue DCT}}]{=} 0.
 \end{align*}
Thus $\left\{\Phi_{1+\frac{1}{n}}\circ v\right\}$  converges to
$v^{+}$ in $W_{b,A,T}(Q)$ and $v^{+}=w$. Therefore
$\left\{\Phi_{1+\frac{1}{n}}\circ v\right\}$  converges to $v^{+}$ in
$W_{B,A,T}(Q)$.

 We get the remainder of $(i)$ by replacing $v$ by $-v$ in the above
 arguments and noting that $|v|=v^{+}-v^{-}$.

\item[$(ii)$] Let $\{v_{n}\}$ be a sequence in $C^{1}(A,\Omega)$ such
  that it converges to  $v$ in $W_{B,A,T}(Q)$. By Lemma \ref{2w1ba},
  $\left\{\Phi_{1+\frac{1}{k}}\circ v_{m}\right\}_{m}$ converges to
  $\Phi_{1+\frac{1}{k}}\circ v$ in  $W_{B,A,T}(Q)$ for every $k$ in
  $\NN$. By the proof of $(i)$, $\left\{\Phi_{1+\frac{1}{k}}\circ v \right\}_{k}$
  converges to $v^{+}$ in  $W_{B,A,T}(Q)$. It implies, for each  $n$
  in $\NN$, there are integers  $k_{n}$ and $m_{n}$ such that
  \begin{align*}
    \norm{\Phi_{1+\frac{1}{k_{n}}}\circ v- v^{+}}_{B,T}
    &<\frac{1}{2n},\\
    \norm{\Phi_{1+\frac{1}{k_{n}}}\circ
    v_{m_{n}}-\Phi_{1+\frac{1}{k_{n}}}\circ v}_{B,T}
    &<\frac{1}{2n}.
  \end{align*}
  It implies
  \[
    \norm{\Phi_{1+\frac{1}{k_{n}}}\circ v_{m_{n}}- v^{+}}_{B,T} \le
    \norm{\Phi_{1+\frac{1}{k_{n}}}\circ
      v_{m_{n}}-\Phi_{1+\frac{1}{k_{n}}}\circ
      v}_{B,T}+\norm{\Phi_{1+\frac{1}{k_{n}}}\circ v-
      v^{+}}_{B,T}<\frac{1}{n}. \] Thus
  $\left\{\Phi_{1+\frac{1}{k_{n}}}\circ v_{m_{n}}\right\}_n$ converges
  to $v^{+}=v$ in $W_{B,A,T}(Q)$. Moreover,
  $\Phi_{1+\frac{1}{k_{n}}}\circ v_{m_{n}} \in C^1(A,\Omega)$ and
  $\Phi_{1+\frac{1}{k_{n}}}\circ v_{m_{n}}\ge 0$. Therefore, we obtain
  $(ii)$ with $w_{n}=\Phi_{1+\frac{1}{k_{n}}}\circ v_{m_{n}}$ for
  every $n$ in $\NN$.
  
\item[$(iii)$] For all $u\in
      W_{B,A,T}(Q),$
  \begin{align}   \label{b0}
    \begin{split}
      &\left\{\int_{Q}|u|^{2}(1+\overline{b})dz\right\}^{\frac{1}{2}}\\
      &\coex[\text{H\"older},~ \eqref{c6}]{\le}  \left\{\int_{Q} \left(1+\overline{b} \right)^{\frac{r}{r-2}}dz\right\}^{\frac{r-2}{2r}}\left\{\int_{Q}|u|^{r}dz\right\}^{\frac{1}{r}}
\\
&\coex[\eqref{a12}]{\le}\left\{\int_{Q} \left(1+\overline{b} \right)^{\frac{r}{r-2}}dx\right\}^{\frac{r-2}{2r}}C(r,\Omega,A)T^{\frac{2\overline{t}-2}{\overline{2t}}}\left\{\int_{Q}
  b^{-\frac{\bar{t}}{2-\bar{t}}}dz\right\}^{\frac{2-\overline{t}}{\overline{2t}}}\norm{u}_{B,T}\\
      &\coex{=}
        C(r,\Omega,A)T^{\frac{2\overline{t}-2}{\overline{2t}}}\norm{1
        + \bar{b}}_{L^{\frac{r}{r-2}}(Q)}^{\frac{1}{2}}
        \norm{b^{-1}}_{L^{\frac{\bar{t}}{2-\bar{t}}}(Q)}^{\frac{1}{2}} \norm{u}_{B,T}.
    \end{split}
  \end{align}
Since $\eta\in C^{1}_{c}(\RR^{N})$, there is a positive real number $M$ such that $|\eta|+|\nabla \eta|\le M $. Let  $\left\{v_{n}\right\}$ be a sequence in $C^{1}(Q,A)$ such that $\left\{v_{n}\right\}$ converges to $v$ in $W_{B,A,T}(Q)$. Then $\eta v_{n}$ is in  $C^{1}(Q,A)$ for every $n$ in $\NN$. Arguing as in \eqref{b1}, we have
 \begin{align*}
   &\norm{\eta v_{n}- \eta v_{m}}_{B,T}^{2}\\
   &\coex[\eqref{vuong}]{=}
     \int_{Q}\vuong \left(\eta v_{n}- \eta v_{m} \right)dz \\
   &\coex[\text{Schwarz}]{\le} 2\int_{Q} \left[(v_{n}-v_{m})^{2}\vuong \eta +
     \eta^{2}\vuong (v_{n}-  v_{m}) \right]dz\\
   &\coex{\le}2 \int_{Q} \left[M^{2} \left(v_{n}-v_{m} \right)^{2} \left(1
     +\overline{b} \right)+ M^{2}\vuong \left(v_{n}-  v_{m} \right)
     \right]dz\\
   &\coex[\eqref{c1}]{\le}
     2M^{2}\int_{Q} \left[(v_{n}-v_{m})^{2} \left(1+\overline{b} \right) + \vuong (v_{n}-
     v_{m}) \right]dz\\
   &\coex[\eqref{b0}]{\le} 2M^{2} \left[C(r,\Omega,A)^{2}T^{\frac{2\overline{t}-2}{\overline{t}}}\norm{1
        + \bar{b}}_{L^{\frac{r}{r-2}}(Q)}
        \norm{b^{-1}}_{L^{\frac{\bar{t}}{2-\bar{t}}}(Q)} +1 \right]\norm{v_{n}-v_{m}}^{2}_{B,T}
 \end{align*}
 for every $m,n\in\NN$.
 
 Thus $\left\{\eta v_{n}\right\}$ is a Cauchy sequence in
 $W_{B,A,T}(Q)$. Then it converges to some $w$ in $W_{B,A,T}(Q)$. By \eqref{cb}, $\left\{\eta v_{n}\right\}$ converges to $w$ in $W_{b,A,T}(Q)$. We shall prove $w=\eta v$. We have
 \begin{align*}
 &\left\{\int_{Q}\abs{\nabla \left(\eta v_{n}-\eta v
   \right)}^{2}bdz\right\}^{\frac{1}{2}}\\
   &\coex{=} \left\{\int_{Q} \left[\left(v_{n}-v
     \right)^{2}\abs{\nabla\eta}^{2}+\eta^{2}\abs{\nabla
     (v_{n}-v)}^{2} \right]bdz\right\}^{\frac{1}{2}}\\
 &\coex[\text{Minkowski}]{\le}
   M \left[\left\{\int_{Q}(v_{n}-v)^{2}(1+b)dz\right\}^{\frac{1}{2}}
 +\left\{\int_{Q}|\nabla(v_{n}-v)|^{2}bdz\right\}^{\frac{1}{2}}
   \right]\\
 &\coex[\eqref{c1} \text{ and }\eqref{b0}]{\le} M
   \left[C(r,\Omega,A)^{2}T^{\frac{2\overline{t}-2}{\overline{t}}}\norm{1
        + \bar{b}}_{L^{\frac{r}{r-2}}(Q)}
        \norm{b^{-1}}_{L^{\frac{\bar{t}}{2-\bar{t}}}(Q)} +1 \right]\norm{v_{n}-v}^{2}_{B,T}
 \end{align*} 
 for every $n\in\NN$.
 
 Thus $\left\{\eta v_{n}\right\}$ converges to $\eta v$ in
 $W_{b,A,T}(Q)$ and hence, $w=\eta v$. Therefore, we get $(iii)$.
 \end{itemize}
 \end{proof}
 \section{Auxiliary functions}\label{aux}
 
 In \cite{serrinLocalBehaviorSolutions1964}, Serrin obtained the local
 boundedness of solutions of quasi-linear equations. In this section,
 we shall use Serrin's method to get the global boundedness of
 solutions of a class of parabolic equations. However, these technique
 are only applicable to local regularity. Therefore, we need to modify
 the auxiliary functions in \cite{serrinLocalBehaviorSolutions1964} as
 follows.
\begin{definition} Let  $s \in (1,\infty)$ and $l \in
  [3,\infty)$. Let
  $$\eta_{s,l}\coloneqq (1-s^{2})l^{s},\quad a_{s,l}\coloneqq \frac{1}{2}s(s+1)l^{s-1},\quad b_{s,l}\coloneqq \frac{1}{2} s(s-1)l^{s+1},$$  and 
  \begin{align}
    F_{s,l}(t)
    &\coloneqq
    \begin{cases}
|t|^{s}&\text{ if } |t| ~\le l,\\
\eta_{s,l}+a_{s,l}|t| +b_{s,l}|t|^{-1}&\text{ if } |t| > l.
\end{cases}
\label{f1b}\\
G_{s,l}&\coloneqq F_{s,l}F'_{s,l}.
\label{f1b2}
\end{align}
\label{def61}
\end{definition}

We have the following results.
\begin{lemma}
  Let $s \in (1,\infty)$ and $l \in [3,\infty)$. Then
  \begin{itemize}
  \item[$(i)$]  $F_{s,l}$ and $G_{s,l}$ are in $C^{1}(\RR)$.
  \item[$(ii)$] $F'_{s,l}$ and $G'_{s,l}$ are in $L^{\infty}(\RR)$ and
  uniformly continuous on $\RR$ and
  \begin{align}
\abs{tF'_{s,l}(t)}
    &\le 4sF_{s,l}(t)\qquad \forall~t\in\RR,
  \label{ss1}\\
  \abs{F'_{s,l}}^{2}
  &\le s^{2}G'_{s,l}(t)\qquad \forall~t\in \RR,
  \label{ss2}\\
  \abs{G_{s,l}(t)}
  &= sF_{s,l}
^{2-\frac{1}{s}}(t)
\qquad ~\text{if }|t|\le 1,
\label{ss3}\\
F_{s,l}(t)
&\le F_{s,k}(t)
\qquad \forall~s\in(1,\infty), t\in\RR,~ l<k.
\label{ss4}
\end{align}
\end{itemize}

\label{lem62}
\end{lemma}
\begin{proof} 
\begin{itemize}[wide]
\item[$(i)$] Note that
  \begin{align}
    F'_{s,l}(t)
    &= \begin{cases}
      s.\sign t\abs{t}^{s-1} &\text{ if } \abs{t} \le l,\\
      \sign t \left( a_{s,l} -b_{s,l}\abs{t}^{-2}  \right)& \text{ if } l<|t|,
    \end{cases}
      \label{f12b}\\
    F''_{s,l}(t)
    &= \begin{cases}
      s(s-1)|t|^{s-2} &\text{ if } |t| \le l,\\
      2b_{s,l}|t|^{-3}&\text{ if } l<|t|,
    \end{cases}
      \label{f12c}\\
    G_{s,l}(t)
    &= s\left(\sign t \right)|t|^{2s-1}\qquad \text{ if } \abs{t} < l.
      \label{f12d}\\
    G'_{s,l}(t)
    &= \left(F'_{s,l} \right)^{2}(t)+F_{s,l}F''_{s,l}(t)\qquad \text{ for all } t \in \RR.
      \label{f12f}\\
    G'_{s,l}(t)
    &= s(2s-1)\abs{t}^{2(s-1)}\qquad \text{ if } \abs{t} < l.
      \label{f12e}
  \end{align}
 Since   $s>1$, by \eqref{f1b}, \eqref{f12b}, \eqref{f12d} and \eqref{f12e}, $F_{s,l}\vert_{(-2,2)}$ and $G_{s,l}\vert_{(-2,2)}$ are of class $C^{1}((-2,2))$. By definitions of $a_{s,l}$, $b_{s,l}$ and $\eta_{s,l}$, we have
  \begin{align}
    \lim_{t\to l^{-}}F_{s,l}''(t)
    &=s(s-1)l^{s-2}=2b_{s,l}\abs{l}^{-3}=\lim_{t\to l^{+}}F_{s,l}''(t),
  \label{f8}\\
  \lim_{t\to l^{-}}F_{s,l}'(t)
  &= sl^{s-1}= \frac{1}{2}s(s+1)l^{s-1}-   \frac{1}{2}s(s-1)l^{s-1}= a_{s,l}-b_{s,l}l^{-2}=\lim_{t\to l^{+}}F_{s,l}'(t),
    \label{f9}
    \intertext{and,}
    \begin{split}
  \lim_{t\to l^{-}}F_{s,l}(t)
      &= l^{s}\\
      &=(1-s^{2})l^{s}+\frac{1}{2}s(s+1)l^{s} + \frac{1}{2}s(s-1)l^{s}
  \\
  &=\eta_{s,l}+a_{s,l}l +b_{s,l}l^{-1}=\lim_{t\to l^{+}}F_{s,l}(t).
    \end{split}
      \label{f10}
\end{align}
By \eqref{f12b}--\eqref{f10}, $F'_{s,l}\vert_{\RR\setminus(-1,1)}$ and
$G'_{s,l}\vert_{\RR\setminus (-1,1)}$ are of class
$C(\RR\setminus(-1,1))$. Thus $F_{s,l}$ and $G_{s,l}$ are in
$C^{1}(\RR)$.
\item[$(ii)$] By \eqref{f12b} and \eqref{f12c}, $F'_{s,l}$ is
  uniformly continuous and bounded on $\RR$. Thus, by \eqref{f12f},
  $G'_{s,l}$ is uniformly continuous and bounded on $\RR$.
  
  Let $t\in \RR\setminus [-l,l]$. Put $x=l^{-1}|t|$. By definition, we
  have
   \begin{align}
     \begin{split}
       a_{s,l}-b_{s,l}|t|^{-2}
       &\coex{=}\frac{1}{2}s(s+1)l^{s-1}-\frac{1}{2}
         s(s-1)l^{s+1}|t|^{-2}\\
       &\coex[l|t|^{-1} \le 1]{\ge}
         \frac{1}{2}sl^{s-1}[(s+1)-(s-1)] ,
     \end{split}
  \label{f11}
   \end{align}
   thus 
   \begin{align*}
     &4sF_{s,l}(t)-\abs{tF'_{s,l}(t)}\\
     &\coex[\eqref{f11}]{=}
       4s\eta_{s,l}
       +4sa_{s,l}|t|+4sb_{s,l}|t|^{-1}-a_{s,l}|t|+b_{s,l}|t|^{-1}\\
     &\coex{=} 4s\left(1-s^{2} \right)l^{s}+2s^{2}(s+1)l^{s-1}|t|+
       2s^{2}(s-1)l^{s+1}|t|^{-1}\\
     &\coex{~}\quad -\frac{1}{2}s(s+1)l^{s-1}|t| +\frac{1}{2}
       s(s-1)l^{s+1}|t|^{-1}\\
     &\coex{=}
       4s\left(1-s^{2} \right)l^{s}+\left(2s^{2}-\frac{1}{2}s \right)(s+1)l^{s-1}|t|+\left(2s^{2}+\frac{1}{2}s \right)(s-1)l^{s+1}|t|^{-1}\\
     &\coex{=} l^{s+1}|t|^{-1}\left\{4s\left(1-s^{2} \right)l^{-1}|t|
       +\left(2s^{2}-\frac{1}{2}s \right)(s+1)l^{-2}t^{2}
       +\left(2s^{2}+\frac{1}{2}s \right)(s-1)\right\}\\
     &\coex{=} l^{s+1}|t|^{-1}f(x),
   \end{align*}
where $x \coloneqq l^{-1}\abs{t}$ and
\[ f(z)\coloneqq \left(2s^{2}-\frac{1}{2}s
  \right)(s+1)z^{2}+4s\left(1-s^{2} \right)z+\left(2s^{2}+\frac{1}{2}s
  \right)(s-1)\qquad\forall~z\in \RR.
\]
Let $x\in [1,\infty)$. By computation, we get
\begin{align*}
  f'(x)
  &\coex{=} 2\left(2s^{2}-\frac{1}{2}s \right)(s+1)x+4s\left(1-s^{2}
    \right)\\
  &\coex[x\ge 1]{\ge} 2\left(2s^{2}-\frac{1}{2}s \right)(s+1)+4s(1-s^{2})\\
  &\coex{=} 3s^{2}+3s \coex[s>1]{>} 0 \qquad \forall~x\in [1,\infty),
    \intertext{and}
    f(1)
  &\coex{=}    (2s^{2}-\frac{1}{2}s)(s+1)+4s(1-s^{2})+(2s^{2}+\frac{1}{2}s)(s-1)=3s>0.
\end{align*}
It follows that $f(x)>0$ for every $x$ in $[1,\infty)$ and
\begin{equation}
\abs{tF'_{s,l}(t)}\le 4sF_{s,l}(t)\qquad\forall~s\in(1,\infty),~ t\in \RR\setminus [-l,l].
  \label{ff11}
\end{equation}
On the other hand
\begin{equation}
  \abs{tF'_{s,l}(t)}= s \abs{t}^{s} =sF_{s,l}(t)\qquad\forall~ t\in [-l,l].
  \label{f11b}
\end{equation}
Combining \eqref{ff11} and \eqref{f11b}, we get \eqref{ss1}.
  
By \eqref{ss1} and \eqref{f12c}, $F_{s,l}$ and $F''_{s,l}$ are
non-negative on $\RR$. It implies
\begin{equation}
  \abs{F'_{s,l}(t)}^{2}
  \coex[\eqref{f12f}]{\le} G'_{s,l}(t)\qquad\forall~ t\in \RR
  \label{f11bb}
\end{equation}
and we get \eqref{ss2}.

We have
\[
  \abs{G_{s,l}(t)}=s|t|^{2s -1}=s|t|^{s\left(2 -\frac{1}{s} \right)}=sF_{s,l}(t)^{2 -\frac{1}{s}}
   \qquad\text{ if }|t|\le 1, \]
which implies \eqref{ss3}.
 
Fix $s\in(1,\infty)$, $l\in [3,\infty)$ and $k \in(l,\infty)$.  Put
$x\coloneqq l^{-1}t\ge 1$ for every $t\in [l,k]$, and
\begin{align}
  g(t)
  &\coloneqq t^{s}- \left(1-s^{2} \right)l^{s}-\frac{1}{2}s(s+1)l^{s-1}t - \frac{1}{2} s(s-1)l^{s+1}|t|^{-1}\qquad\forall~t\in [l,k],
    \label{lk0}\\
  h(x) &\coloneqq sx^{s-1}-\frac{1}{2}s(s+1)+\frac{1}{2}s(s-1)x^{-2}\qquad\forall~x\in[1,\infty).
\label{lk0b}
 \end{align}
 We have
 \begin{align*}
   &h(1)=0,\\
   &h'(x)=s(s-1)x^{s-2} -s(s-1)x^{-3}
=s(s-1)x^{-3}(x^{s+1}-1) \ge 0\qquad\forall~x\in[1,\infty), 
 \end{align*}
 so
 \begin{equation}
 h(x) \ge  0 \qquad\forall~x\in[1,\infty).
  \label{lk3}
\end{equation}
On the other hand,
\begin{align}
  g(l)
  &\coex{=} 0,
    \label{lk4}\\
  \begin{split}
     g'(t)
  &\coex{=} st^{s-1}-\frac{1}{2}s(s+1)l^{s-1}+\frac{1}{2}
    s(s-1)l^{s+1}|t|^{-2}\\
  &\coex[\eqref{lk0b}]{=}l^{s-1}h(x) \coex[\eqref{lk3}]{\ge} 0\qquad\forall~t\in [l,k].
  \end{split}
\label{lk5}
 \end{align}
Combining \eqref{lk4} and \eqref{lk5}, by  the Mean Value Theorem, we
have 
\begin{align*}
  g(t)
  &= g(t) - g(l)
    = g'(c_{t,l}) (t-l) \ge 0,
\end{align*}
so
\[
  t^{s}\ge \left(1-s^{2} \right)l^{s}+\frac{1}{2}s(s+1)l^{s-1}t + \frac{1}{2} s(s-1)l^{s+1}|t|^{-1}\qquad\forall~t\in [l,k]
\]
or 
\begin{equation}
  F_{s,k}(t)\ge F_{s,l}(t)\qquad\forall~t\in [l,k].
\label{sla}
\end{equation}

Fix $t$ in $[k,\infty)$. Put
\[\overline{h}(\xi)\coloneqq F_{s,\xi}(t)=(1-s^{2})\xi^{s}+ \frac{1}{2}s(s+1)\xi^{s-1}t+\frac{1}{2}s(s-1)\xi^{s+1}t^{-1}\hk\forall~\xi\in[l,k].\]
Since $\xi^{-1}t++\xi t^{-1} \ge 2$, we have 
\begin{align*}
  \overline{h}'(\xi)
  &= s\left(1-s^{2} \right)\xi^{s-1}+ \frac{1}{2}s\left(s^{2}-1
    \right)\xi^{s-2}t+\frac{1}{2}s\left(s^{2}-1 \right)\xi^{s}t^{-1}\\
  &= s\left(s^{2}-1 \right)\xi^{s-1} \left[-1+ \frac{1}{2} \left(\xi^{-1}t+\xi t^{-1} \right) \right]\ge 0 \qquad\forall~\xi\in[l,k].
\end{align*}
This  implies
\begin{equation}F_{s,k}(t)=\overline{h}(k)\ge  \overline{h}(l) =F_{s,l}(t)\qquad\forall~t\in [k,\infty).
\label{slb}
\end{equation}
Combining \eqref{sla} and \eqref{slb}, we get \eqref{ss4}.
\end{itemize}
\end{proof}

For the case  $s\in (\frac{1}{2},1]$, we have following auxiliary functions.
\begin{definition}
  Let $s \in \left(\frac{1}{2},1 \right]$. Put
  \begin{align}
    \theta(t)
    &\coloneqq \begin{cases}
      \frac{3}{8} \abs{t}^{\frac{1}{2}} \left(\abs{t}^{2}-\frac{10}{3}\abs{t} +5
      \right)&\text{ if } \abs{t} \le 1,\\
      1 &\text{ if } 1 <\abs{t},
\end{cases}
\label{fs1}\\
\overline{\theta}_{s}(t)
&\coloneqq \abs{t}^{s}\qquad\forall~t\in\RR ,
\label{fs2}\\
F_{s}
&\coloneqq \theta\overline{\theta}_{s},
\label{fs3}\\
G_{s}
&\coloneqq F_{s}F'_{s}.
\label{fs4}\\
\overline{F}(t)
&\coloneqq \begin{cases}
0&\text{ if } \abs{t} \le 1,\\
\abs{t} &\text{ if } 1 <\abs{t}.
\end{cases}
\label{fs1z}
\end{align}
\label{defs1}
\end{definition}

We shall use the commands \texttt{D[f,t]}, \texttt{Expand[expr]}, \texttt{Collect[expr,t]} of the
software \texttt{Mathematica} in computations. In order to verify a equation,
we can use the commands  \texttt{Expand} and \texttt{Collect} to the most complicated
side to get the remain side of the equation. Of course we can do these
jobs by hand.  We have the following results.

\begin{lemma} Let $s\in \left(\frac{1}{2},1 \right]$. We have 
\begin{itemize}
\item[$(i)$] $F_{s}$  is in $C^{1}(\RR)$ and $F'_{s}$ is in
  $L^{\infty}(\RR)$ and uniformly continuous on $\RR$,
\item[$(ii)$]  $G_{s}$ is in $C^{1}(\RR)$ and $G'_{s}$ is in
$L^{\infty}(\RR)$ and uniformly continuous on $\RR$.
\end{itemize}
\label{lemfs1}
\end{lemma}
 \begin{proof} 
 \begin{itemize}[wide]
 \item[$(i)$] We have
\begin{subequations}
\begin{align}
  \theta'(t)
  &\coex[\mathtt{D[.]},~\mathtt{Expand[.]}]{=}
    \frac{3}{8} \left( \sign t \right) \abs{t}^{-\frac{1}{2}} \left(\frac{5}{2}t^{2}-5 \abs{t}+\frac{5}{2} \right)\qquad~if~|t|\le 1,
    \label{fs500}\\
  \theta''(t)
  &\coex[\mathtt{D[.]}]{=}
    \frac{3}{8}
    \abs{t}^{-\frac{3}{2}} \left(\frac{15}{4}t^{2}-\frac{5}{2}t-\frac{5}{4}
    \right)\qquad \text{ if } \abs{t}\le 1,
    \label{fs50}\\
  \theta \left(\RR\setminus(-1,1) \right)
  &\coex{=}  \{1\},
    \label{fs5}\\
  \theta'\left(\RR\setminus(-1,1) \right)
  &\coex{=}\theta''\left(\RR\setminus(-1,1) \right) \coex{=} \{0\},
    \label{fs6}\\
  \overline{\theta}_{s}'(t)
  &\coex{=} s \left( \sign t \right) \abs{t}^{s-1}\qquad\forall~t\in \RR\setminus\{0\},
\label{fs6b}\\
    \overline{\theta}_{s}''(t)
  &\coex{=} s(s-1)\abs{t}^{s-2}\qquad\forall~t\in \RR\setminus\{0\},
    \label{fs6c}
    \intertext{By definition,}
  F_{s}(t)
  &\coex{=}\theta\overline{\theta}_{s}(t) \qquad\forall~t\in\RR,
    \label{fs6z}\\
  F_{s}(t)
  &\coex{=}
    \frac{3}{8} \abs{t}^{s+\frac{1}{2}} \left(\abs{t}^{2}-\frac{10}{3}\abs{t}+5 \right)\qquad
    \text{ if } \abs{t}\le 1,
    \label{fs6zb}\\
  F_{s}'(t) &\coex{=}\theta'\overline{\theta}_{s}(t)+\theta\overline{\theta}_{s}'(t)\qquad\forall~t\in \RR,
     \label{fs7zb}\\
  \begin{split}
     F_{s}'(t)
  &\coex{=} \left[\frac{3}{8} \abs{t}^{s+\frac{1}{2}}
    \left(\abs{t}^{2}-\frac{10}{3}\abs{t} +5 \right) \right]'\\
  &\coex[\substack{\mathtt{D[.]},~
    \mathtt{Expand[.]},\\\mathtt{Collect[.,t]}}]{=}
  \frac{3}{8} \left( \sign t \right) \abs{t}^{s-\frac{1}{2}} \left[
    \left(\frac{5}{2} + s \right) t^2 - \left(5 + \frac{10s}{3}
    \right)\abs{t} + \frac{5}{2} + 5s  \right]\\
  &\coex{~} \qquad \quad \qquad \text{ if } \abs{t}\le 1,
  \end{split}
  \label{fs7zbb}\\
  F_{s}''(t)
  &\coex{=}
    \theta''\overline{\theta}_{s}(t)+2\theta'\overline{\theta}_{s}'(t)+\theta\overline{\theta}_{s}''(t)\qquad\forall~t\in
    \RR,
    \label{fs8z}\\
\begin{split}
  F_{s}''(t)
  &\coex{=}
    \left\{\frac{3}{8} \left( \sign t \right) \abs{t}^{s-\frac{1}{2}}
    \left[ \left(\frac{5}{2} + s \right) t^2 - \left(5 + \frac{10s}{3} \right)\abs{t} + \frac{5}{2} +
    5s  \right]  \right\}'\\
  &\coex[\substack{\mathtt{Expand},\\ \mathtt{Collect[.,t]}}]{=}
    \frac{3}{8}\abs{t}^{s-\frac{3}{2}} \left[\left(\frac{15}{4}+ 4 s +
  s^2 \right) t^2 - \left(\frac{5}{2} + \frac{20 s}{3} + \frac{10
  s^2}{3} \right) t - \frac{5}{4} + 5 s^2 \right]\\
  &\coex{~} \qquad \quad \qquad \text{ if } \abs{t}\le 1.
\end{split} \label{fs8zb}
\end{align}
\end{subequations}
We have
\begin{subequations}
  \label{eq:1}
\begin{align}
  \lim_{t\to 1^{-}}\theta(t)
  &\coex[\eqref{fs1}]{=} 1 \coex{=} \lim_{t\to 1^{+}}\theta(t),
  \label{fs9z}\\
  \lim_{t\to 1^{-}}\theta'(t)
  &\coex[\eqref{fs500}]{=} 0  \coex{=} \lim_{t\to 1^{+}}\theta'(t),
  \label{fs10z}\\
  \lim_{t\to 1^{-}}\theta''(t)
  &\coex[\eqref{fs50}]{=} 0  \coex{=} \lim_{t\to 1^{+}}\theta''(t).\label{fs11z}
\end{align}
\end{subequations}
It implies $\theta\in C^{2}(\RR)$. Thus, by \eqref{fs6z} and \eqref{eq:1},
$F_s$ is in $C^{1}(\RR)$ and $F'_{s}$ is
in $L^{\infty}(\RR)$ and uniformly continuous on $\RR$ if $s\in
\left(\frac{1}{2},1 \right]$.

\item[$(ii)$] Let $s\in \left(\frac{1}{2},1 \right]$. We have
  \begin{subequations}
    \label{eq:2}
    \begin{align}
      G_{s}(t)
      &\coex{=} \left(F_{s}F_{s}' \right)(t)\\
      &\coex{=} \left( \sign t \right)
        \abs{t}^{2s} \left[s \abs{t}^{-1}\theta^{2}(t) + \left( \sign
        t \right) \theta\theta'(t) \right]\qquad\forall~t\in\RR,\\
      &\coex{=}
        \begin{cases} \displaystyle
\frac{9}{64} \left( \sign t \right) \abs{t}^{2s}
          \left(\abs{t}^{2}-\frac{10}{3}\abs{t} +5 \right)\times\\
          \qquad \displaystyle \times \left[ \left(\frac{5}{2} + s \right) t^2 - \left(5 + \frac{10s}{3} \right)\abs{t} + \frac{5}{2} + 5s  \right]
          &\text{ if } \abs{t} \le 1,\vspace{1em}\\
          s \left( \sign t \right) \abs{t}^{2s-1} & \text{ if } 1 <\abs{t},
        \end{cases}\\
      \begin{split}
        G_{s}'(t)
      &\coex{=} \left(F_{s}'
        \right)^{2}(t)+F_{s}F_{s}''(t)\qquad\forall~t\in\RR,\\
      &\coex{=} \abs{t}^{2s-2} \left[s^{2}+s(s-1) \right]
      \coex{=} s(2s-1)\abs{t}^{2s-2}~ \text{ if } 1 < \abs{t},
      \end{split}
      \label{fs12z}\\
      \begin{split}
        G_{s}'(t)
        &\coex{=}
          \left\{\frac{9}{64} \left( \sign t \right) \abs{t}^{2s}
          \left(\abs{t}^{2}-\frac{10}{3}\abs{t} +5 \right)
          \right. \times \\
        &\coex{~} \qquad \quad \times \left.\left[ \left(\frac{5}{2} +
          s \right) t^2 - \left(5 + \frac{10s}{3} \right)\abs{t} +
          \frac{5}{2} + 5s  \right]\right\}'\\
        &\coex[\substack{\mathtt{D[\cdot]},~\mathtt{Expand},\\
        \mathtt{Collect}}]{=}
        \frac{9}{64} \abs{t}^{2 s -1} \left[\left(10 + 9 s + 2 s^2
        \right) t^4 - \left(40 + \frac{140 s}{3} + \frac{40
        s^2}{3} \right)\abs{t}^3\right.\\
        &\coex{~} \qquad \quad
          + \left(\frac{190}{3} + \frac{950 s}{9}+ \frac{380 s^2}{9} \right)t^2\\
        &\coex{~} \qquad \quad
          -\left.
          \left(\frac{100}{3} + 100 s + \frac{200 s^2}{3} \right)\abs{t} +
          25 s + 50 s^2\right]\quad \text{ if } \abs{t}\le 1.
      \end{split}\label{fs12zb}
    \end{align}
  \end{subequations}
Since $0<2s-1\le 1$, by \eqref{eq:1} and \eqref{eq:2}, $G_{s}$ is in
$C^{1}(\RR)$ and $G'_{s}$ is  in $L^{\infty}(\RR)$ and uniformly
continuous on $\RR$.
\end{itemize}
\end{proof}
\begin{lemma}  There are  positive real numbers  $\delta\in \left(0,\frac{1}{3} \right)$ and $c_{0}$ such that
 \begin{equation} 
\abs{F'_{s}}^{2}\le c_{0} G'_{s}(t)\qquad \forall~s\in \left(\frac{2}{3}-\delta,1 \right],~ t\in \RR,
\label{sss2}
\end{equation}
\label{lemfs1b}
\end{lemma}
\begin{proof}
  Let $(\alpha,s,t)\in (0,1]\times \left(\frac{1}{2},1 \right]\times [0,1]$.  Put
 \begin{equation}h(\alpha,s,t)\coloneqq \left( \frac{9}{64}t^{2s -1}
   \right)^{-1} \left[\alpha \left(F_{s}' \right)^{2}(t)+F_{s}F_{s}''(t) \right].
\label{hh}
\end{equation}
We prove this lemma in following steps.
\begin{description}[wide]
\item[Step 1.]
For all $t \in \RR$,
\begin{align}
\label{0z}
  \begin{split}
    &2t^4 - \frac{40}{3}t^3 + \frac{380}{9}t^2 - \frac{200}{3}t + 50\\
    &\coex[\mathtt{Expand}]{=}
      2\left(t^2 - \frac{10}{3}t+1 \right)^2 + 16\left(t- \frac{80}{3\cdot16} \right)^{2}
      -16\left(\frac{5}{3} \right)^{2} + 48>3.
  \end{split}
\end{align}
Fix $t$ in $[0,1]$. Put $k(s)\coloneqq h(1,s,t)$ for every $s$ in $\left(\frac{1}{2},1 \right]$. Then
\begin{align}
  \label{0zaz}
  \begin{split}
    k(s)
    &\coex[\eqref{hh}]{=}\left(  \frac{9}{64}t^{2s -1}
      \right)^{-1}G_{s}'(t)\\
    &\coex[\eqref{fs12zb}]{=} \left(10 + 9 s + 2 s^2 \right) t^4 -
      \left(40 + \frac{140 s}{3} + \frac{40 s^2}{3} \right)t^3 \\
    &\coex{~}\quad + \left(\frac{190}{3} + \frac{950 s}{9}+ \frac{380
      s^2}{9} \right)t^2\\
    &\coex{~} \quad - \left(\frac{100}{3} +   100 s + \frac{200
      s^2}{3} \right)t + 25 s + 50 s^2\\
    &\coex[\substack{\mathtt{Expand[.]},\\ \mathtt{Collect[.,s]}}]{=}
    \left(50 - \frac{200 t}{3} + \frac{380 t^2}{9}- \frac{40 t^3}{3}+
    2 t^4 \right)s^2 \\
    &\coex{~} \quad + \left(25 - 100 t + \frac{950 t^2}{9}- \frac{140
      t^3}{3}+9 t^4 \right)s\\
    &\coex{~} \quad-\frac{100 t}{3} + \frac{190 t^2}{3} - 40 t^3 + 10 t^4,
  \end{split}\\
  \label{0zazb}
  \begin{split}
    k'(s)
    &\coex[\mathtt{D[.]}]{=} 2 \left(2t^4 - \frac{40}{3}t^3 +
      \frac{380}{9}t^2 - \frac{200}{3}t + 50 \right)s\\
    &\coex{~} \quad + \left(9 t^4 - \frac{140}{3}t^3 +
      \frac{950}{9}t^2 - 100t+25 \right)\\
    &\coex[\eqref{0z},~ s>\frac{1}{2}]{\ge}
      \left(2t^4 - \frac{40}{3}t^3 + \frac{380}{9}t^2 - \frac{200}{3}t
      + 50 \right)\\
    &\coex{~} \quad+ \left(9 t^4 - \frac{140}{3}t^3 + \frac{950}{9}t^2
      - 100t+25 \right)\\
    &\coex[\substack{\mathtt{Expand[.]},\\\mathtt{Collect[.,t]}}]{=}
    11 t^{4}- 60 t^3 + \frac{1330}{9}t^2- \frac{500}{3}t  + 75\\
    &\coex[\substack{\mathtt{Expand[.]},\\\mathtt{Collect[.,t]}}]{=}
    11\left(t^{2}- \frac{30}{11}t + 2 \right)^2 +\frac{2174 }{99} \left(t -
    \frac{99\cdot 70}{3\cdot 2174} \right)^2 + \frac{6747}{1087}\\
    &\coex{>} 0.
  \end{split}
\end{align}
It implies
\[0< k'(s)\le 10^{3} \qquad\forall~s\in \left(\frac{1}{2},1 \right].
\]
 Thus, by the Mean Value Theorem, $k$ is increasing on $\left(\frac{1}{2},1 \right]$ and
 \begin{equation}
   0 < k(s)-k(s')\le 10^{3}(s-s')\qquad \text{ if }~\frac{1}{2}<s'<s\le 1.
\label{s}
\end{equation}
\item[Step 2.] Let $t\in [0,1]$, by direct computation, we obtain
\begin{align*}
  h\left(1,\frac{2}{3},t \right)
  &\coex{=} k\left(\frac{2}{3} \right)\\
  &\coex[\eqref{0zaz},~\mathtt{Expand}]{=}\frac{152}{9}t^4 -
    \frac{2080}{27}t^3  +  \frac{12350}{81}t^2- \frac{3500}{27}t +
    \frac{350}{9}\\
  &\coex[\mathtt{Expand}]{=}
    \frac{152}{9} \left(t^2 - \frac{1040}{456}t +\frac{5}{4} \right)^2 +
    \frac{3830}{171} \left(t - \frac{285}{383} \right)^2  -
    \frac{3830}{171} \left(\frac{285}{383} \right)^2 + \frac{225}{18}\\
  &\coex{>} 10^{-2}.
\end{align*}
 Thus 
\begin{equation}
h \left(1,\frac{2}{3},t \right) > 10^{-2}\qquad\forall~t\in [0,1].
\label{wwzb}
\end{equation}

\item[Step 3.] Fix $t\in [0,1].$ Let $\delta > 0$ to be chosen later.
From \eqref{s}, we have 
\begin{align*}
h \left( 1, \frac{2}{3}, t \right) - h \left( 1, \frac{2}{3}-\delta, t
  \right)
  \le 10^3 \delta,
\end{align*}
or 
\begin{align*}
  h \left( 1, \frac{2}{3}-\delta,t \right)
  &\coex{\ge} h \left( 1, \frac{2}{3}, t \right) - 10^3 \delta\\ &\coex[\eqref{wwzb}]{>} 10^{-2} - 10^3\delta
\end{align*}
From \eqref{0zazb}, $h \left( 1, \cdot, t \right)$ is increasing on
$\left( \frac{1}{2}, 1 \right].$ By choosing $\delta = 10^{-6}$, we have 
\begin{align}
 h(1,s,t)> h \left( 1, \frac{2}{3}-\delta,t \right) > 10^{-3}\qquad\forall~(s,t)\in \left(\frac{2}{3}-\delta,1 \right]\times[0,1].
 \label{h1}
\end{align}
\item[Step 4.] For any $t\in [-1,1]$ and $\alpha \in [0,1]$, we have
\begin{align*}
  \frac{\partial h(\alpha,s,t)}{\partial \alpha}
  &\coex[\eqref{hh}]{=}
    \left( \frac{9}{64}t^{2s -1} \right)^{-1} \left(F_{s}'
    \right)^{2}(t)\\
  &\coex[\eqref{fs7zbb}]{=} \left[ \left(\frac{5}{2} + s \right) t^2 -
    \left(5 + \frac{10s}{3} \right)\abs{t} + \frac{5}{2} + 5s
    \right]^{2}\\
  &\coex{\le} 10^4. 
\end{align*}
Choosing $\alpha_{0} =1-10^{-8}$, by \eqref{h1}  and the Mean Value Theorem, we get
\begin{align}
  \begin{split}
  h(\alpha_{0},s,t)&\coex{\ge} h(1,s,t) - 10^{-4}\\
  &\coex[\eqref{h1}]{>}10^{-4}\qquad\forall~(s,t)\in (\frac{2}{3}-\delta,
    1]\times [0,1].
  \end{split}
\label{alpha}
\end{align}

\item[Step 5.] Put $\beta=\frac{1}{1-\alpha_{0}}$. By computation, we
  have 
  \begin{align*}
    \beta G_{s}'(t)-\left(F_{s}'(t) \right)^{2}
    &\coex{=} \left((\beta-1)F'^{2}+ \beta FF'' \right)(t)\\
    &\coex{=} \beta \left(\alpha_{0} F'^{2}+ FF'' \right)(t)\\
    &\coex[\eqref{hh}]{=}
      \frac{9\beta}{64}t^{2s-1}h(\alpha_{0},s,t)\\
    &\coex[\eqref{alpha}]{>}\frac{9 t^{2s-1}}{10^{4}64(1-\alpha_{0})}
      \coex{\ge} 0
  \end{align*}
  or
  \[
  \left(F_{s}'(t) \right)^{2}\le \beta G_{s}'(t)\qquad\forall~(s,t)\in \left(\frac{2}{3}-\delta,1 \right]\times [0,1].\]
Since $\left( F_s' \right)^2$ and $G_s'$ are even functions, we obtain
\[
  \left(F_{s}'(t) \right)^{2}\le \beta G_{s}'(t)\qquad\forall~(s,t)\in \left(\frac{2}{3}-\delta,1 \right]\times [-1,1].\]
 On the other hand
\begin{align*}
  2G_s'(t)- \left(F_s'(t) \right)^{2}
  &\coex{=} \left(F_s'(t) \right)^{2}+2\left(F_sF_s'' \right)(t)\\
  &\coex{=} s^{2}t^{2s-2}+ s(s-1)t^{2s-2}\\
  &\coex{=} st^{2s-2}(2s-1)\\
  &\coex[s > \frac{1}{2}]{\ge} 0 \qquad \text{ if } \abs{t} \ge 1, s> \frac{1}{2}.
\end{align*}
Putting  $c_{0}\coloneqq \max \left\{\frac{1}{1-\alpha_{0}},2\right\}$, we get the lemma.
\end{description}
\end{proof}

\begin{lemma} Let  $\delta$ be in Lemma \ref{lemfs1b}.  There are positive real numbers  $k_{s}$  and $k_{0}$ such that
\begin{align}
  \abs{tF'_{s}(t)}
  &\le 5F_{s}(t)\qquad\forall~s\in\left(\frac{2}{3}-\delta,1 \right],~t\in\RR,
\label{sss1}\\
\abs{G_s(t)}
&\le k_{s}F_s^{2-\frac{1}{s}}(t)\qquad \text{if } \abs{t}\le 1,
 \label{sss3}\\
 \overline{F}(t)
 &\le F_{s}^{\frac{1}{s}}(t)\le \overline{F}(t)+k_{0}\qquad \forall~t\in \RR.
 \label{sss4}
\end{align}
\label{lemfs1c}
\end{lemma}
\begin{proof}
  By computation, we obtain 
  \begin{itemize}[wide]
  \item If $\abs{t} \le 1$:
  \begin{align*}
    &5F_{s}(t)-\abs{tF_{s}'(t)}\\
    &\coex[\eqref{fs7zbb}]{=}
      \frac{3}{8} \abs{t}^{s+\frac{1}{2}}
      \left[5\left(\abs{t}^{2}-\frac{10}{3}\abs{t} +5
      \right)-\left(\left(\frac{5}{2} + s \right) t^2 - \left(5 +
      \frac{10s}{3} \right)\abs{t} + \frac{5}{2} + 5s  \right)
      \right]\\
    &\coex[\substack{\mathtt{Expand},\\\mathtt{Collect[.,t]}}]{=}
    \frac{3}{8} \abs{t}^{s+\frac{1}{2}} \left(\frac{5-2s}{2}
    t^{2}-\frac{35-10s}{3}\abs{t}+\frac{45-10s}{2} \right)\\
    &\coex[\substack{0<s\le 1,\\  \abs{t}\le 1}]{\ge}
    \frac{3}{8} \abs{t}^{s+\frac{1}{2}}\left(
    -\frac{35}{3}+\frac{35}{2} \right) \\
    &\coex{\ge} 0.
  \end{align*} 
\item If $\abs{t} \ge 1$: 
\begin{align*}
5F_{s}(t)-\abs{tF_{s}'(t)}=(5-s)\abs{t}^{s}\coex[0<s\le 1]{\ge} 0.
\end{align*}
\end{itemize}
Thus we get \eqref{sss1}.

By definition \eqref{fs4},
\begin{align*}
  \abs{G_{s}(t)}
  &=\abs{\theta\overline{\theta}_{s}(t)
    \left[\overline{\theta}_{s}'\theta+\overline{\theta}_{s}\theta'
    \right](t)}\\
  &= \abs{\theta\overline{\theta}_{s}}^{2-\frac{1}{s}}(t)
    \abs{\overline{\theta}_{s}'\theta^{\frac{1}{s}}\overline{\theta}_{s}^{\frac{1}{s}-1}+\overline{\theta}_{s}^{\frac{1}{s}}\theta'\theta^{\frac{1}{s}-1}}(t)\\
  &= \abs{F_s}^{2-\frac{1}{s}}(t)
    \abs{\overline{\theta}_{s}'\theta^{\frac{1}{s}}\overline{\theta}_{s}^{\frac{1}{s}-1}+\overline{\theta}_{s}^{\frac{1}{s}}\theta'\theta^{\frac{1}{s}-1}}(t)
\end{align*}
 Since  $\theta$, $\theta'$, $\overline{\theta}_{s}$ and
  $\overline{\theta}_{s}'\theta^{\frac{1}{s}}\overline{\theta}_{s}^{\frac{1}{s}-1}
  = s \left( \sign t \right) \theta^{\frac{1}{s}}$ are continuous on $[-1,1]$, and $\frac{1}{s}-1\ge 0$, there is a positive real number $k_{s}$ such that
\begin{align*}
  \abs{G_s(t)}
  &\le k_s \abs{F_s}^{2-\frac{1}{s}}(t)
    \coex[\eqref{sss1}]{=} k_s F_s^{2-\frac{1}{s}}(t) \qquad \text{ if } \abs{t} \le 1.
\end{align*}
Thus we get \eqref{sss3}.

Note that
\begin{align*}
  \overline{F}(t)
  &= F_{s}^{\frac{1}{s}}(t)\qquad \text{ if } \abs{t}\ge 1,\\
  \overline{F}(t)
  &=0\le F_{s}^{\frac{1}{s}}(t)\qquad \text{ if } \abs{t}\le 1,\\
  F_{s}^{\frac{1}{s}} (t)
  &=  \abs{t}^{1+\frac{1}{2s}} \left[\frac{3}{8}
    \left(\abs{t}^{2}-\frac{10}{3}\abs{t} +5 \right)
    \right]^{\frac{1}{s}}\\
  &\le \left[\frac{3}{8} \left(1+\frac{10}{3} +5 \right)
    \right]^{\frac{1}{s}}\\
  &\le{k_{0}}\qquad \text{ if } \abs{t}\le 1,~ s>\frac{1}{2},
\end{align*}
where $k_{0}= \left[\frac{3}{8} \left(1+\frac{10}{3} +5 \right) \right]^{2}$. This implies \eqref{sss4}.
\end{proof}
\section{Boundedness for  solutions of parabolic equations}\label{bo}

We use techniques in \cite{ducGlobalHolderContinuity,moserHarnackInequalityParabolic1964} to study  the boundedness of parabolic equations. We need a following lemma.
\begin{lemma}
  Let $A$ be admissible with respect to $\Omega$. Assume \eqref{eq:ccc}
  holds.
  \begin{itemize}
  \item[$(i)$] Suppose $u\in V_{B,A,T}(Q)$.  Let $F\coloneqq F_{s,l}$
    or $F_{s}$, and $G\coloneqq G_{s,l}$ or $G_{s}$ defined in Definitions
    \ref{def61} or \ref{defs1}, respectively. Let $w\coloneqq u^{+}$,
    $v\coloneqq F\circ u^{+}$, $\varphi\coloneqq G\circ u^{+}$ and
    $\eta\in C^{1} \left(\RR^{N+1} \right)$ such that
    $\eta(Q)\subset [0,1]$. We have
  \begin{align}
    \begin{split}
      &\frac{1}{2}c_{0}\frac{\partial (\eta^{2}v^{2})}{\partial
      t}+\sum_{i,j=1}^{N}\frac{\partial (\eta v)}{\partial
        x_{i}}\frac{\partial(\eta v)}{\partial x_{j}}b_{ij}\\
      &\coex{\le} c_{0} \left[
    \eta^{2}\varphi\frac{\partial u}{\partial
      t}+\sum_{i,j=1}^{N}\frac{\partial (\eta^{2}\varphi)}{\partial
      x_{i}}\frac{\partial u}{\partial x_{j}}b_{ij} \right]\\
    &\coex{~} \qquad + c_{0}\eta v^{2}\abs{\nabla \eta}+2c_{0} v
      \abs{\nabla_{x}(\eta v)} \abs{\nabla_x \eta}\overline{b}+ 3c_{0}v^{2}\abs{\nabla_x \eta}^{2}\overline{b},
    \end{split}
    \label{p1031}
  \end{align}
where  $c_{0}$ is defined in Lemma \ref{lemfs1b}.

\item[$(ii)$]  
  If  $u\in W_{B,A,T}(Q)$ satisfying
  \begin{align}
    \begin{split}
      &\int_{Q}\left( \frac{\partial u}{\partial t}\psi +
        \sum_{i,j=1}^{N}\frac{\partial u}{\partial x_{i}}\frac{\partial
        \psi}{\partial x_{j}}b_{ij}  \right)dz\\
      &\coex{\le}\int_{Q} \left( a_{0} \abs{\nabla_{x} u }+a_{1}\abs{u}  +a_{2}
        \right) \abs{\psi} dz \qquad\forall~\psi\in W_{B,A,T}(Q),
    \end{split}
 \label{le460}
 \end{align}
then
\begin{align}
  \label{p103}
  \begin{split}
    \left\{\int_{Q}\eta^{r}v^{r}dz\right\}^{\frac{2}{r}}
    &\le C(r,Q,A)^2\int_{Q}\sum_{i=1}^{N}\abs{\frac{\partial (\eta v)}{\partial x_{i}}}^{2}bdz\\
    &\le C(r,Q,A)^2\int_{Q}\sum_{i,j=1}^{N}\frac{\partial (\eta v)}{\partial x_{i}}\frac{\partial(\eta v)}{\partial x_{j}}b_{ij}dz\\
    &\le C(r,Q,A)^2 \left[ c_0\int_{Q}\left(  a_{0}\abs{\nabla_{x} u}+a_{1}\abs{u}+a_{2} \right)\eta^{2}\varphi dz \right. \\
    &\qquad+  c_{0}\int_{Q}\eta v^{2}\abs{\nabla \eta} dz +  2 \int_{Q}v
      \abs{\nabla_{x}(\eta v)} \abs{\nabla_x \eta}\overline{b}dz\\
    &\qquad+  \left.3c_{0}\int_{Q}v^{2}\abs{\nabla_x \eta}^{2}\overline{b}dz \right].
  \end{split}
\end{align}
Here, $C(r,Q,A)$ is the constant in Definition \ref{d41}.
\end{itemize}
\label{lem12}
\end{lemma}
\begin{proof}
  \begin{itemize}[wide]
  \item[$(i)$] Since $F(0)=G(0)=0$, by Lemma \ref{2w1ba} and Lemma \ref{w1bb}, $w$,  $v$ and $\varphi$ are in $V_{B,A,T}(Q)$.  Let $j\in \{1,\cdots,N+1\}$. We have 
\begin{align}
  \frac{\partial w}{\partial x_{j}}(x)
  &\coex{=} \begin{cases}
\displaystyle\frac{\partial u}{\partial x_{j}}(x) \qquad&{\forall~x \in Q^{+}=\{z\in Q :u(z)>0\},}\\
0&{\forall~x \in Q\setminus Q^{+}}.
\end{cases}
    \label{le462aazz20}\\
  \label{le462aazz2}
  \begin{split}
    \frac{\partial v}{\partial x_{j}}(x)
    &\coex{=} \begin{cases}
      \displaystyle F'(w(x))\frac{\partial u}{\partial x_{j}}(x) \qquad&{\forall~x \in Q^{+},}\\
      0&{\forall~x \in Q\setminus Q^{+}},
\end{cases}
    \\
    &\coex[F'(0)=0]{=} F'(w(x))\frac{\partial w}{\partial
      x_{j}}(x)\qquad\forall~x \in Q.
  \end{split}\\
  \label{le462aazz3}
  \begin{split}
    \frac{\partial \varphi}{\partial x_{j}}(x)
  &\coex{=} \begin{cases}
    \displaystyle G'(w(x))\frac{\partial u}{\partial x_{j}}(x)  \qquad&{\forall~x \in Q^{+},}\\
    0&{\forall~x \in Q\setminus Q^{+},}
       \end{cases}
        \\
&\coex[G_{l}'(0)=0]{=} G_{l}'(w(x))\frac{\partial w}{\partial
  x_{j}}(x)\qquad\forall~x \in Q.
  \end{split}
\end{align}
By results in
\cite[pp.150-151]{gilbargEllipticPartialDifferential2001} and direct
computation, we obtain
\begin{align}
  \eta^{2}\varphi\frac{\partial u}{\partial t}
  \coex[G=FF']{=} \eta^{2}F(u^{+})F'(u^{+})\frac{\partial u}{\partial t}=\frac{1}{2}\eta^{2}\frac{\partial v^{2}}{\partial t}=\frac{1}{2}\frac{\partial (\eta^{2}v^{2})}{\partial t}-\eta v^{2}\frac{\partial\eta}{\partial t},
\label{pa1}
\end{align}
and
\begin{align*}
  &\sum_{i,j=1}^{N}\frac{\partial \left(\eta^{2}\varphi \right)}{\partial
    x_{i}}\frac{\partial u}{\partial x_{j}}b_{ij}\\
  &\coex[\varphi=G \left(u^{+} \right)]{=}
    \eta^{2}G'(u^{+})\sum_{i,j=1}^{N}\frac{\partial u^{+}}{\partial
    x_{i}}\frac{\partial u}{\partial x_{j}}b_{ij}+ 2\eta
    G(u^{+})\sum_{i,j=1}^{N}\frac{\partial \eta}{\partial
    x_{i}}\frac{\partial u}{\partial x_{j}}b_{ij}\\
  &\coex[\eqref{ss2} \text{ and }\eqref{sss2}]{\ge}
    c_{0}^{-1}\eta^{2}\sum_{i,j=1}^{N} \left(F'(u^{+}) \right)^{2}\frac{\partial
    u^{+}}{\partial x_{i}}\frac{\partial u^{+}}{\partial x_{j}}b_{ij}+
    2\eta F(u^{+})F'(u^{+})\sum_{i,j=1}^{N}\frac{\partial
    \eta}{\partial x_{i}}\frac{\partial u}{\partial x_{j}}b_{ij}\\
  &\coex{=} c_{0}^{-1}\eta^{2}\sum_{i,j=1}^{N}\frac{\partial
    F(u^{+})}{\partial x_{i}}\frac{\partial F(u^{+})}{\partial
    x_{j}}b_{ij}+ 2\eta F(u^{+})\sum_{i,j=1}^{N}\frac{\partial
    \eta}{\partial x_{i}}\frac{\partial F(u^{+})}{\partial
    x_{j}}b_{ij} \\
  &\coex{=} c_{0}^{-1}\eta^{2}\sum_{i,j=1}^{N}\frac{\partial
    v}{\partial x_{i}}\frac{\partial v}{\partial x_{j}}b_{ij}+ 2\eta
    v\sum_{i,j=1}^{N}\frac{\partial \eta}{\partial
    x_{i}}\frac{\partial v}{\partial x_{j}}b_{ij} \\
  &\coex[b_{ij}=b_{ji}]{=}
    c_{0}^{-1}
\sum_{i,j=1}^{N}\frac{\partial (\eta v)}{\partial x_{i}}\frac{\partial(\eta v)}{\partial x_{j}}b_{ij}-2c_{0}^{-1} v\sum_{i,j=1}^{N}\frac{\partial  (\eta v)}{\partial x_{i}}\frac{\partial\eta }{\partial x_{j}}b_{ij} + c_{0}^{-1}
v^{2}\sum_{i,j=1}^{N}\frac{\partial  \eta}{\partial
    x_{i}}\frac{\partial\eta }{\partial x_{j}}b_{ij}\\
  &\coex{~} \qquad
    + 2v\sum_{i,j=1}^{N}\frac{\partial \eta}{\partial x_{i}}\frac{\partial (\eta v)}{\partial x_{j}}b_{ij} - 2v^{2}\sum_{i,j=1}^{N}\frac{\partial \eta}{\partial x_{i}}\frac{\partial \eta }{\partial x_{j}}b_{ij}.
\end{align*}
The latter implies
\begin{align}
  \begin{split}
    \sum_{i,j=1}^{N}\frac{\partial (\eta v)}{\partial
  x_{i}}\frac{\partial(\eta v)}{\partial x_{j}}b_{ij}
&\le
c_{0}\sum_{i,j=1}^{N}\frac{\partial (\eta^{2}\varphi)}{\partial
  x_{i}}\frac{\partial u}{\partial x_{j}}b_{ij} -(2c_{0}-2)
v\sum_{i,j=1}^{N}\frac{\partial  (\eta v)}{\partial
  x_{i}}\frac{\partial\eta }{\partial x_{j}}b_{ij}\\
&\qquad \quad+ (2c_{0}-1)v^{2}\sum_{i,j=1}^{N}\frac{\partial  \eta}{\partial x_{i}}\frac{\partial\eta }{\partial x_{j}}b_{ij}
  \end{split}
\label{pa1b}
\end{align}

Combining \eqref{pa1} and \eqref{pa1b}, we get 
\begin{align*}
\sum_{i,j=1}^{N}\frac{\partial (\eta v)}{\partial
    x_{i}}\frac{\partial(\eta v)}{\partial x_{j}}b_{ij}
  &\coex[\eqref{pa1} \text{ and } \eqref{pa1b}]{\le}
    c_{0} \left[\eta^{2}\varphi\frac{\partial u}{\partial
    t}+\sum_{i,j=1}^{N}\frac{\partial \left(\eta^{2}\varphi \right)}{\partial
    x_{i}}\frac{\partial u}{\partial x_{j}}b_{ij} \right] \\
  &\coex{~}\qquad -
    c_{0} \left[\frac{1}{2}\frac{\partial \left(\eta^{2}v^{2} \right)}{\partial t}-\eta
    v^{2}\frac{\partial\eta}{\partial t} \right]\\
  &\coex{~} \qquad -2(c_{0}-1
) v\sum_{i,j=1}^{N}\frac{\partial  (\eta v)}{\partial
    x_{i}}\frac{\partial\eta }{\partial x_{j}}b_{ij}\\
  &\coex{~} \qquad +
    (2c_{0}-1)v^{2}\sum_{i,j=1}^{N}\frac{\partial \eta}{\partial
    x_{i}}\frac{\partial \eta }{\partial x_{j}}b_{ij}\\
  &\coex[\eqref{c1},~ \text{Schwarz}]{\le}
    c_{0} \left[ \eta^{2}\varphi\frac{\partial u}{\partial
    t}+\sum_{i,j=1}^{N}\frac{\partial \left(\eta^{2}\varphi \right)}{\partial
    x_{i}}\frac{\partial u}{\partial x_{j}}b_{ij} \right]\\
  &\coex{~}\qquad -
    \frac{1}{2}c_{0}\frac{\partial \left(\eta^{2}v^{2} \right)}{\partial t}+
    c_{0}\eta v^{2}\abs{\nabla \eta}\\
  &\coex{~} \qquad +2c_{0} v \abs{\nabla_x (\eta v)} \abs{\nabla_x \eta}\overline{b}+ 3c_{0}v^{2}\abs{\nabla_x \eta}^{2}\overline{b},
\end{align*}
and we obtain $(i)$.
\item[$(ii)$] Let $\{u_{n}\}$ be a sequence in $C^{1}(Q,A)$ that
  converges to $u$ in $V_{B,A,T}(Q)$. Define 
  \begin{align*}
  v \coloneqq F \circ u^+ \qquad \text{ and } \qquad v_n \coloneqq F
    \circ u_n^+.
  \end{align*} We have
  \begin{align}
    \begin{split}
  \int_{Q}\frac{\partial (\eta^{2}v_{n}^{2})}{\partial t}dz
  &\coex[\text{Fubini}]{=}
  \int_{\Omega} \left[\int_{0}^{T}\frac{\partial
  (\eta^{2}v_{n}^{2})}{\partial t}dt \right]dx\\
  &\coex[v_{n} \left(\Omega\times\{0\} \right) =\{0\}]{=}
      \int_{\Omega}(\eta^{2}v_{n}^{2})(x,T)dx\ge 0.
    \end{split}
 \label{p102z1}
  \end{align}
  
By definitions, Lemmas \ref{4l2b}, \ref{2w1ba}, and \ref{w1bb}, $\left\{\eta^{2}G(u_{n}^{+})\right\}$, $\displaystyle\left\{\frac{\partial u_{n}}{\partial t}\right\}$ and $\left\{ F(u_{n}^{+})\right\}$ converge to $\eta^{2}\varphi$, $\displaystyle \frac{\partial u}{\partial t}$ and $ v$ in $L^{2}(Q)$. Thus, by \eqref{pa1}, $\displaystyle \left\{\frac{\partial(\eta^{2}F^{2}(u_{n}^{+}))}{\partial t}\right\}$ converges to $\displaystyle \left\{\frac{\partial(\eta^{2}v^{2})}{\partial t}\right\}$ in $L^{1}(Q)$. Thus by \eqref{p102z1}, we obtain
 \begin{equation}
\int_{Q}\frac{\partial (\eta^{2}v^{2})}{\partial t}dz\ge 0.
 \label{p102z2}
 \end{equation}
 Thus we obtain 
 \begin{align*}
   &C(r,Q,A)^{-2}\left\{\int_{Q}\eta^{r}v^{r}dz\right\}^{\frac{2}{r}}\\
   &\coex[\eqref{cb}]{\le}
     \int_{Q}\sum_{i=1}^{N}\abs{\frac{\partial (\eta v)}{\partial
     x_{i}}}^{2}bdz\\
   &\coex[\eqref{c1}]{\le}
     \int_{Q}\sum_{i,j=1}^{N}\frac{\partial (\eta v)}{\partial
     x_{i}}\frac{\partial(\eta v)}{\partial x_{j}}b_{ij}dz\\
   &\coex[\eqref{p1031}]{\le}
     c_{0}\int_{Q} \left( \eta^{2}\varphi\frac{\partial u}{\partial t}
     +\sum_{i,j=1}^{N}\frac{\partial (\eta^{2}\varphi)}{\partial
     x_{i}}\frac{\partial u}{\partial x_{j}}b_{ij} \right)dz -\frac{c_{0}}{2}
     \int_{Q}\frac{\partial \left(\eta^{2}v^{2} \right)}{\partial t}dz\\
   &\coex{~} \qquad + c_{0}\int_{Q}\eta v^{2}\abs{\nabla \eta} dz
     +2\int_{Q}v \abs{\nabla_{x}(\eta v)} \abs{\nabla_x \eta}\overline{b}dz
     +3c_{0}\int_{Q}v^{2}\abs{\nabla_x \eta}^{2}\overline{b}dz\\
   &\coex[\eqref{le460} \text{ and }\eqref{p102z2}]{\le}
     c_{0}\int_{Q} \left( a_{0}|\nabla_{x} u|+a_{1}|u|
     +a_{2} \right)\eta^{2}\varphi dz\\
    &\coex{~} \qquad + c_{0}\int_{Q}\eta v^{2}\abs{\nabla \eta} dz
     +2\int_{Q}v \abs{\nabla_{x}(\eta v)} \abs{\nabla_x \eta}\overline{b}dz
     +3c_{0}\int_{Q}v^{2}\abs{\nabla_x \eta}^{2}\overline{b}dz.
 \end{align*}
and we obtain $(ii)$.
\end{itemize}
\end{proof}
 We have following global boundedness of solutions  as follows. 
 \begin{proposition}[Global boundedness]
Let $A$ be  admissible with respect to $\Omega$ and assume
\eqref{eq:ccc} holds. Let $\delta$ be defined as in Lemma
\ref{lemfs1b}, and $\alpha \in [1,\infty)$ such that 
\begin{equation}
\frac{\alpha}{\overline{r}}> \frac{2}{3}-\delta.
\label{cc9}
\end{equation}
Let $0 < a\in L^{\frac{\overline{r}}{\overline{r}-2}}(Q)$, and $a_{0}$, $a_{1}$, $a_{2}$ be
non-negative measurable functions on $Q$ such that
\begin{equation}
b^{-1} a_{0}^{2}+a_{1}+ a_{2}\le a,
\label{cc8}
 \end{equation}
Suppose $u\in V_{B,A,T}(Q)\cap L^{\alpha}(Q)$ and satisfies  \eqref{le460}.

Then there exists a real number
$c\left(r,\alpha,Q,\norm{u}_{L^{\alpha}(Q)} \right) > 0$ such that
\begin{equation}
  \norm{u}_{L^{\infty}(Q)}\le c \left(r,\alpha,\Omega,\norm{u}_{L^{\alpha}(Q)} \right). \label{le460abc}
 \end{equation}
 \label{pro46}
\end{proposition}
  \begin{proof} 
    Let $\kappa \coloneqq \frac{r}{\overline{r}}$, then $\kappa > 1$
    by definition. For $m \in \NN$, define
 \begin{align*}
   s_{m}\coloneqq \kappa^{m}\frac{\alpha}{\overline{r}}\ge
   \frac{\alpha}{\overline{r}}
   \qquad \text{ and } \qquad
   \overline{s}_{m}\coloneqq \kappa^{m}\alpha.
 \end{align*}
 \begin{description}[wide]
 \item[Case 1.] \emph{$s_m > 1$ for all $m \in \NN$.}
   
   Fix a non-negative integer $m$.  Let $\eta=1$, $l\in [3,\infty)$,
   $F \coloneqq F_{s_{m},l}$ and $G\coloneqq G_{s_{m},l}$ be defined
   as in Definition \ref{def61} and let
   \begin{align}
     \label{eq:5}
   \varphi_{s_m,l} \coloneqq G \circ u^+, \quad v_{s_m,l} \coloneqq F \circ u^+, \quad
     \text{ and }\quad w \coloneqq u^+.
   \end{align}
   We temporarily suppress the subscripts to lighten the notation,
   that is, we write $\varphi$ and $v$ instead of
   $\varphi_{s_m,l}$ and $v_{s_m,l}$, respectively.
   
   Because $F(0)=G(0)=0$, by Lemma \ref{2w1ba} and Lemma \ref{w1bb}, $w$,  $v,$ and $\varphi$ are in $W_{B,A,T}(Q)$.
   Since \[ \frac{\frac{1}{s_{m}}}{\overline{r}}+\frac{\overline{r}-2}{\overline{r}}+\frac{2-\frac{1}{s_{m}}}{\overline{r}}=1,
   \]
by H\"{o}lder's inequality, we have
\begin{align}
  \int_{\Omega}F(w)^{2-\frac{1}{s_{m}}}a dx
  \le \abs{\Omega}^{\frac{1}{\overline{r}s_{m}}}\norm{a}_{L^{\frac{\overline{r}}{\overline{r}-2}}(\Omega)}\left\{\int_{\Omega}F(w)^{\overline{r}} dx\right\}^{\frac{2-\frac{1}{s_{m}}}{\overline{r}}}.
 \label{c}
 \end{align}
 Using \eqref{p103} with $\eta=1$, we have 
\begin{align*}
  \begin{split}
      &\int_{Q}\sum_{i=1}^{N} \abs{\frac{\partial v}{\partial
    x_{i}}}^{2}bdz\\
    &\coex[\eqref{p103}]{\le}
      c_{0}\int_{Q}
     \left(a_{0}\abs{\nabla_{x} u}+a_{1}\abs{u} +a_{2}
      \right)G(u^{+})dz\\
    &\coex[\eqref{f1b2},\eqref{ss3},\eqref{cc8}]{\le}
      c_{0}\int_{Q}F(w)\abs{\nabla_{x}
      w}\abs{F'(w)}b^{\frac{1}{2}}a^{\frac{1}{2}}dz+
      c_{0}\int_{Q}\abs{w}F(w)\abs{F'(w)} a dz\\
    &\coex{~}\quad + c_{0}\int_{w> 1} \abs{w}F(w)\abs{F'(w)} a dx +
      s_{m}c_{0}\int_{w\le 1} F(w)^{2-\frac{1}{s_{m}}}a dx\\
    &\coex[\text{Young},~\eqref{ss1}]{\le}
      \frac{1}{2} \int_{Q}\abs{\nabla_{x}
      w}^{2}F'(w)^{2}bdz+\frac{c_{0}^{2}}{2 }
      \int_{Q}F(w)^{2}adz\\
    &\coex{~} \quad
      +8s_{m}c_{0}  \int_{Q} v^{2} a dz+
      s_{m}c_{0}\int_{w\le 1} v^{2-\frac{1}{s_{m}}}a dz\\
    &\coex[\substack{\text{H\"{o}lder},~\eqref{c}}]{\le}
      \frac{1}{2}\int_{Q}\abs{\nabla_{x} v}^{2}bdz+\frac{c_{0}^{2}}{2}\norm{a}_{L^{\frac{\overline{r}}{\overline{r}-2}}(Q)}\left\{\int_{Q}F(w)^{\overline{r}}
      dz\right\}^{\frac{2}{\overline{r}}}\\
    &\coex{~} \quad +
      8s_{m}c_{0}\norm{a}_{L^{\frac{\overline{r}}{\overline{r}-2}}(Q)}\left\{\int_{Q}F(w)^{\overline{r}}
      dz\right\}^{\frac{2}{\overline{r}}}\\
    &\coex{~}\quad +s_{m}c_{0}^{2}\abs{Q}^{\frac{1}{\overline{r}s_{m}}}\norm{a}_{L^{\frac{\overline{r}}{\overline{r}-2}}(Q)}\left\{\int_{Q}F(w)^{\overline{r}}
      dz\right\}^{\frac{2-\frac{1}{s_{m}}}{\overline{r}}}\\
    &\coex[s_m > 1,~\eqref{le462aazz2}]{\le}
      \frac{1}{2} \int_{Q}|\nabla_{x}
      v|^{2}bdz+9s_{m}c_{0}^{2}\norm{a}_{L^{\frac{\overline{r}}{\overline{r}-2}}(Q)}\left\{\int_{Q}v^{\overline{r}}
      dz\right\}^{\frac{2}{\overline{r}}}\\
    &\coex{~} \quad
      +
      s_{m}c_{0}^{2}\abs{Q}^{\frac{1}{\overline{r}s_{m}}}\norm{a}_{L^{\frac{\overline{r}}{\overline{r}-2}}(Q)}\left\{\int_{Q}v^{\overline{r}}
      dz\right\}^{\frac{2-\frac{1}{s_{m}}}{\overline{r}}}\\
    &\coex[2-\frac{1}{s_{m}}\le 2]{\le}
      \frac{1}{2} \int_{Q}\abs{\nabla_x v}^{2}bdz +
      9s_{m}c_{0}^{2} \left[\left(1+\abs{Q}^{\frac{1}{\overline{r}s_{m}}} \right)\norm{a}_{L^{\frac{\overline{r}}{\overline{r}-2}}(Q)} \right]\max\left\{1,\left\{\int_{Q} v^{\overline{r}}dz\right\}^{\frac{2}{\overline{r}}}\right\}
  \end{split}
\end{align*} 
  or
  \[
\int_{Q}\abs{\nabla_x v}^{2} bdz \le  18s_{m}c_{0}^{2}\left[\left(\abs{Q}+2 \right)\norm{a}_{L^{\frac{\overline{r}}{\overline{r}-2}}(Q)} \right]\max\left\{1,\left\{\int_{Q}v^{\overline{r}}dz\right\}^{\frac{2}{\overline{r}}}\right\}
  \]
Thus
\[
  \left\{ \int_{Q}v^{r}dz \right\}^{\frac{2}{r}}
 \coex[\eqref{cb}]{\le} C_{1}^{2}s_{m}\max\left\{1,\left\{\int_{Q}v^{\overline{r}}dz\right\}^{\frac{2}{\overline{r}}}
  \right\}
 \] 
  or 
\begin{align*}\left\{\int_{Q}v^{r} dz\right\}^{\frac{1}{r}}
    \coex[s_{m}\ge 1]{\le}  C_{1}s_{m}\max\left\{1,\left\{\int_{Q}v^{\overline{r}}dz\right\}^{\frac{1}{\overline{r}}}\right\}\qquad\forall~m=0,1,2,\cdots
\end{align*}
where
\[C_{1} \coloneqq \left\{18C(r,Q,A)^2c_{0}^{2} \left(\abs{Q}+2
    \right)\norm{a}_{L^{\frac{\overline{r}}{\overline{r}-2}}(Q)}+1\right\}^{\frac{1}{2}}\ge
  1. \]
Bringing back the subscripts that were suppressed after \eqref{eq:5}, we have
\begin{align}
  \left\{\int_{Q}v_{s_m,l}^{r} dz\right\}^{\frac{1}{r}}
    \coex[s_{m}\ge 1]{\le}  C_{1}s_{m}\max\left\{1,\left\{\int_{Q}v_{s_m,l}^{\overline{r}}dz\right\}^{\frac{1}{\overline{r}}}\right\}\qquad\forall~m=0,1,2,\cdots
\label{le465}
\end{align}
 By \eqref{f1b} and \eqref{ss4}, the sequence
 $\left\{v_{s_{m},3n}\right\}_n$ is  increasing  and converges  to
 $w^{s_{m}}$ a.e. $z \in Q$. Thus by  Lebesgue Monotone Convergence Theorem and \eqref{le465}, we have 
 \[\left\{\int_{Q} w^{s_{m}r}  dz\right\}^{\frac{1}{r}}\le  
 C_{1}s_{m}\max\left\{1,\left\{\int_{Q}w^{s_{m}\overline{r}}  dz\right\}^{\frac{1}{\overline{r}}}\right\}  \qquad\forall~m\ge 0
\]
or
\[\left\{\int_{Q} w^{s_{m}r} dz\right\}^{\frac{1}{s_{m}r}}\le
 C_{1}^{\frac{1}{s_{m}}}s_{m}^{\frac{1}{s_{m}}}\max\left\{1,\left\{\int_{Q}w^{s_{m}\overline{r}}  dz\right\}^{\frac{1}{s_{m}\overline{r}}}\right\}  \qquad\forall~m\ge 0
 \]
Recall $\kappa=\frac{r}{\overline{r}}>1$ and $s_{m}=
\kappa^{m}\frac{\alpha}{\overline{r}}$, so $s_{m}r=
\kappa^{m+1}\alpha$, $s_{m}\overline{r}= \kappa^{m}\alpha$,
$\frac{1}{s_{m}}=\frac{\overline{r}}{\alpha}\frac{1}{\kappa^{m}} $, and 
\begin{align*}
C_{1}^{\frac{1}{s_{m}}} s_{m}^{\frac{1}{s_{m}}}= \left(
  C_{1}^{\frac{\overline{r}}{\alpha}} \right)^{\frac{1}{\kappa^{m}}} \left[ \kappa^{m}\frac{\alpha}{\overline{r}}
\right]^{\frac{\overline{r}}{\alpha}\frac{1}{\kappa^{m}}}
= \left( C_{1}^{\frac{\overline{r}}{\alpha}}
  \right)^{\frac{1}{\kappa^{m}}} \left[\left(\frac{\alpha}{\overline{r}}
  \right)^{\frac{\overline{r}}{\alpha}} \right]^{\frac{1}{\kappa^{m}}}
\left[ \left(\frac{r}{\overline{r}}
  \right)^{\frac{\overline{r}}{\alpha}}
  \right]^{\frac{m}{\kappa^{m}}}
  = c_{1}^{\frac{1}{\kappa^{m}}}c_{2}^{\frac{m}{\kappa^{m}}},
\end{align*}
where $c_{1}\coloneqq
C_{1}^{\frac{\overline{r}}{\alpha}} \left(\frac{\alpha}{\overline{r}} \right)^{\frac{\overline{r}}{\alpha}}$
and $c_{2}\coloneqq\left(\frac{r}{\overline{r}} \right)^{\frac{\overline{r}}{\alpha}}$.
We obtain
\begin{equation}
  \left\{\int_{Q}  w^{\kappa^{m+1}\alpha} dz\right\}^{\frac{1}{\kappa^{m+1}\alpha} }  \le  
  c_{1}^{\frac{1}{\kappa^{m}}}c_{2}^{\frac{m}{\kappa^{m}}}\max\left\{1,\left\{\int_{Q} w^{\kappa^{m}\alpha}  dz\right\}^{\frac{1}{\kappa^{m}\alpha}}\right\}\hk\forall~m\ge 0.
 \label{le465b}
\end{equation}
 By mathematical induction, we get
 \[
   \left\{\int_{Q}w^{\kappa^{n}\alpha}dz\right\}^{\frac{1}{\kappa^{n}\alpha}}
   \le
c_{1}^{\sum_{j=1}^{n}\frac{1}{\kappa^{j}}}c_{2}^{\sum_{j=1}^{n}\frac{j}{\kappa^{j}}}\max\left\{1,\left\{\int_{Q}w^{\alpha}
     dz\right\}^{\frac{1}{\alpha}}\right\}\qquad\forall~n\in\NN.
\]
    Therefore,  by Problem 5 in \cite[p.71]{rudinRealComplexAnalysis1987}, it implies 
    \[
      \norm{u^{+}}_{L^{\infty}(Q)}
      \coex[w = u^+]{\le}      c_{1}^{\sum_{j=1}^{\infty}\frac{1}{\kappa^{j}}}c_{2}^{\sum_{j=1}^{\infty}\frac{j}{\kappa^{j}}}\max\left\{1,\left\{\int_{Q}|u|^{\alpha}
          dz\right\}^{\frac{1}{\alpha}}\right\}.
    \]
    
 Replacing $w=u^{+}$ by $w=u^{-}$.  We have 
 \begin{align*}
   \frac{\partial w}{\partial x_{j}}(x)
   &\coex{=} \begin{cases}
     \displaystyle -\frac{\partial u}{\partial x_{j}}(x) \qquad&{\forall~x \in Q^{-},}\\
     0&{\forall~x \in Q\setminus Q^{-}}.
   \end{cases}\\
   \frac{\partial v}{\partial x_{j}}(x)
   &\coex{=} \begin{cases}
\displaystyle - F'(w(x))\frac{\partial u}{\partial x_{j}}(x) \qquad&{\forall~x \in Q^{-},}\\
0&{\forall~x \in Q\setminus Q^{-}}
   \end{cases}\\
   &\coex[F'(0)=0]{=}
     F'(w(x))\frac{\partial w}{\partial x_{j}}(x)\qquad\forall~x \in
     Q.\\
   \frac{\partial \varphi}{\partial x_{j}}(x)
   &\coex{=}\begin{cases}
\displaystyle - G'(w(x))\frac{\partial u}{\partial x_{j}}(x)  \qquad&{\forall~x \in Q^{-},}\\
0&{\forall~x \in Q\setminus Q^{-}.}
\end{cases}
 \end{align*}
Using $\varphi = G \circ w$ as a test function  in \eqref{p103},   we
have 
\begin{align*}
  \int_{Q}\abs{\nabla_x v}^{2} bdz
  &\coex[\eqref{c1}]{=} \int_{Q}\frac{\partial v}{\partial x_{i}}
    \frac{\partial v}{\partial x_{j}}b_{ij}dz\\
  &\coex{\le} c_{0}\int_{Q}\abs{\nabla_{x} u} \abs{\varphi} a_{0}
    dz+c_{0}\int_{Q} \left(\abs{u}+1 \right)
    \abs{\varphi}a dz
\end{align*}
Now arguing as above, we get
\[
  \norm{u^{-}}_{L^{\infty}(Q)}\leq
  c_{1}^{\sum_{j=1}^{\infty}\frac{1}{\kappa^{j}}}c_{2}^{\sum_{j=1}^{\infty}\frac{j}{\kappa^{j}}}\max\left\{1,\left\{\int_{Q}\abs{u}^{\alpha}
      dz\right\}^{\frac{1}{\alpha}}\right\}.
\]
Thus we obtain the proposition for this case.
\item[Case 2.] \emph{There exists $\overline{m} \in \NN$ such that
    $s_{\overline{m}} \le 1$}.

Since $\kappa > 1$ and $s_m = \kappa^m \frac{\alpha}{\overline{r}}
\ge \frac{\alpha}{\overline{r}}$, this case occurs if and only if
$\frac{\alpha}{\overline{r}} \le 1$ or $\alpha \in [1,\bar{r}]$. Note that $s_m$ is an increasing
sequence and $\lim_{m \to \infty}
s_m = +\infty$, the number $m_{\alpha} \coloneqq \max \left\{ m \colon s_m
  \le 1 \right\}$ is finite.

Applying the argument in Case 1 for the sequence
$s'_m \coloneqq s_{m-m_{\alpha}}$, we obtain
\begin{align}
\label{eq:3}
  \norm{u^{+}}_{L^{\infty}(Q)}
 \le c_{1}^{\sum_{j=m_{\alpha}+1}^{\infty}\frac{1}{\kappa^{j}}}c_{2}^{\sum_{j=m_{\alpha}+1}^{\infty}\frac{j}{\kappa^{j}}}\max\left\{1,\left\{\int_{Q}w^{\left( \kappa^{m_{\alpha}} \right)\alpha}
      dz\right\}^{\frac{1}{\left( \kappa^{m_{\alpha}} \right)\alpha}}\right\}.
\end{align}
Thus the proposition is proved if we can bound $\left\{\int_{Q}w^{\left( \kappa^{m_{\alpha}} \right)\alpha}
      dz\right\}^{\frac{1}{\left( \kappa^{m_{\alpha}} \right)\alpha}}$
    by a multiple of $\left\{ \int_Q w^{\alpha} dz \right\}^{\frac{1}{\alpha}}.$

Fix $\NN \ni m \le m_{\alpha}$. Thanks to \eqref{cc9}, 
\begin{align*}
s_m = \kappa^m \frac{\alpha}{\bar{r}} \coex[\eqref{cc9}]{\in} \left( \frac{2}{3}-\delta,1
  \right] \subset \left( \frac{1}{2}, 1 \right],
\end{align*}
which allows us to utilize the auxiliary functions $F=F_{s_{m}}$,
$G=G_{s_{m}}$ and $c_{0}$ defined in  Definition \ref{defs1} and Lemma
\ref{lemfs1}. Since $F(0)=G(0)=0$, by Lemma \ref{2w1ba} and Lemma
\ref{w1bb}, $w \coloneqq u^{+}$,  $v\coloneqq F\circ w$ and $\varphi\coloneqq G\circ w$ are in $W_{B,A,T}(Q)$.  We have 
 \begin{equation}
 \frac{\partial v}{\partial x_{j}}(x)=  F'(w)\frac{\partial w}{\partial x_{j}},
\label{2le462aazz2}\end{equation}
\begin{equation} \frac{\partial \varphi}{\partial x_{j}}(x)= G'(w)\frac{\partial w}{\partial x_{j}}. 
\label{2le462aazz3}\end{equation}\hk
By choosing $\eta =1$ in \eqref{p103},   we have
\begin{align*}
  \begin{split}
    &\int_{Q}\sum_{i=1}^{N} \abs{\frac{\partial v}{\partial
      x_{i}}}^{2}bdz\\
    &\coex[\eqref{p103}]{\le}
      c_{0}\int_{Q}
      \left(a_{0}\abs{\nabla_{x} u}+a_{1}\abs{u} +a_{2}
      \right)G(u^{+})dz\\
    &\coex[\substack{\eqref{fs4},~\eqref{sss3},\\\eqref{cc8}}]{\le}
      c_{0}\int_{Q}F(w)\abs{\nabla_{x}
      w}\abs{F'(w)}b^{\frac{1}{2}}a^{\frac{1}{2}}dz+
      c_{0}\int_{Q}\abs{w}F(w)\abs{F'(w)} a dz\\[-1em]
    &\coex{~}\quad + 
      c_{0}k_{s_{m}}\int_{w\le 1} F(w)^{2-\frac{1}{s_{m}}}a dx +
      c_{0}\int_{w> 1} \abs{w}F(w)\abs{F'(w)} a dx\\
    &\coex[\substack{\eqref{cc8},~\eqref{c},\\\eqref{sss1},~\text{H\"{o}lder}}]{\le}
    c_{0}\int_{Q} F(w)\abs{\nabla_{x} w}\abs{F'(w)}
    b^{\frac{1}{2}}a^{\frac{1}{2}} dz\\[-1em]
    &\coex{~}\quad +
      c_{0}k_{s_{m}}\abs{Q}^{\frac{1}{\overline{r}s_{m}}}\norm{a}_{L^{\frac{\overline{r}}{\overline{r}-2}}(Q)}\left\{\int_{\abs{w}\le
      1}F(w)^{\overline{r}}
      dz\right\}^{\frac{2-\frac{1}{s_{m}}}{\overline{r}}}+10c_{0}  \int_{Q}
      F(w)^{2} a dz\\
    &\coex[\text{Young}]{\le}
      \frac{1}{2} \int_{Q}\abs{\nabla_{x}
      w}^{2}F'(w)^{2}bdz+\frac{1}{2}c_{0}^{2} \int_{Q}F(w)^{2}adz\\
    &\coex{~}\quad
      +c_{0}k_{s_{m}}\abs{Q}^{\frac{1}{\overline{r}s_{m}}}\norm{a}_{L^{\frac{\overline{r}}{\overline{r}-2}}(Q)}\left\{\int_{\abs{w}\le
      1}F(w)^{\overline{r}}
      dz\right\}^{\frac{2-\frac{1}{s_{m}}}{\overline{r}}} \\
    &\coex{~}\quad+
      10c_{0}\norm{a}_{L^{\frac{\overline{r}}{\overline{r}-2}}(Q)}\left\{\int_{Q}F(w)^{\overline{r}}
      dz\right\}^{\frac{2}{\overline{r}}}\\
    &\coex[\eqref{2le462aazz2}]{\le}
      \frac{1}{2} \int_{Q}\abs{\nabla_{x}v}^{2}bdz +c_{0}(c_{0}+10)\norm{a}_{L^{\frac{\overline{r}}{\overline{r}-2}}(Q)}\left\{\int_{Q}v^{\overline{r}}dz\right\}^{\frac{2}{\overline{r}}}\\
    &\coex{~}\quad +
      c_{0}k_{s_{m}}\abs{Q}^{\frac{1}{\overline{r}s_{m}}}\norm{a}_{L^{\frac{\overline{r}}{\overline{r}-2}}(Q)}\left\{\int_{\abs{w}\le
      1}F(w)^{\overline{r}}
      dz\right\}^{\frac{2-\frac{1}{s_{m}}}{\overline{r}}}\\
    &\coex{\le}
      \frac{1}{2} \int_{Q} |\nabla_{x} v|^{2}bdz \\
    &\coex{~}\quad +
      \left[c_{0}(c_{0}+10) +c_{0}k_{s_{m}}\abs{Q}^{\frac{1}{\overline{r}s_{m}}}\norm{a}_{L^{\frac{\overline{r}}{\overline{r}-2}}(Q)} \right]\max\left\{1,\left\{\int_{Q}v^{\overline{r}}dz\right\}^{\frac{2}{\overline{r}}}\right\},
  \end{split}
\end{align*}
which implies
\[
  \int_{Q}\abs{\nabla_{x} v}^{2} bdz \le 2\left[c_{0}(c_{0}+10)+c_{0}k_{s_{m}}\abs{Q}^{\frac{1}{\overline{r}s_{m}}}\norm{a}_{L^{\frac{\overline{r}}{\overline{r}-2}}(Q)} \right]\max\left\{1,\left\{\int_{Q}v^{\overline{r}}dz\right\}^{\frac{2}{\overline{r}}}\right\}.
  \]
 By \eqref{cb}, we have
\[   \left\{\int_{Q}v^{r} dz\right\}^{\frac{2}{r}}\le
  2C(r,Q,A)^{2}c_{0}
  \left[(c_{0}+10)+k_{s_{m}}\abs{Q}^{\frac{1}{\overline{r}s_{m}}}\norm{a}_{L^{\frac{\overline{r}}{\overline{r}-2}}(Q)}
  \right]\max\left\{1,\left\{\int_{Q}v^{\overline{r}}dz\right\}^{\frac{2}{\overline{r}}}\right\}. \]

Put $M =\max\left\{k_{s_{0}},\cdots,k_{s_{m_{\alpha}+1}}\right\}$ and 
\[
  c_{\alpha}= \left\{ 2C(r,Q,A)^{2}c_{0}
  \left[(c_{0}+10)+M(1+\abs{Q})\norm{a}_{L^{\frac{\overline{r}}{\overline{r}-2}}(Q)}
  \right] \right\}^{\frac{1}{2}}.
\]
Since $s_{m}r= \left(\frac{r}{\bar{r}}
\right)^{m}\frac{\alpha}{\bar{r}}r=\kappa^{m+1}\alpha$ and
$\bar{r}s_{m}= \bar{r} \kappa^m \frac{\alpha}{\bar{r}}=\kappa^{m}\alpha$, we get  
\begin{align}  
  \label{ff}
  \begin{split}
    \left\{\int_{Q}\left( v^{\frac{1}{s_{m}}}
    \right)^{\kappa^{m+1}\alpha}
    dz\right\}^{\frac{1}{\kappa^{m+1}\alpha}}
    &\coex{\le}
      c_{\alpha}^{\frac{1}{\kappa^{m}\alpha}}\max\left\{1,\int_{Q}
      \left(v^{\frac{1}{s_{m}}}
      \right)^{\kappa^{m}\alpha}dz\right\}^{\frac{1}{\kappa^{m}\alpha}}
    \\
    &\coex{\le} c_{\alpha}^{\frac{1}{\kappa^{m}\alpha}}\left\{\int_{Q}
      \left(v^{\frac{1}{s_{m}}} \right)^{\kappa^{m}\alpha}dz\right\}^{\frac{1}{\kappa^{m}\alpha}}+c_{\alpha}^{\frac{1}{\kappa^{m}\alpha}}.
  \end{split}
\end{align} 
Since $v^{\frac{1}{s_{m}}}= F_{s_{m}}(w)^{\frac{1}{s_{m}}}$, we get
\begin{align*}
  \begin{split}
    \left\{\int_{Q}w
    ^{\kappa^{m+1}\alpha} dz\right\}^{\frac{1}{\kappa^{m+1}\alpha}}
    &\coex[\eqref{fs1z}]{\le}
      \left\{\int_{Q} \left[\overline{F}(w)+1 \right]
      ^{\kappa^{m+1}\alpha} dz\right\}^{\frac{1}{\kappa^{m+1}\alpha}}\\
    &\coex[\text{Minkowski}]{\le}
      \left\{\int_{Q}\overline{F}(w)
^{\kappa^{m+1}\alpha} dz\right\}^{\frac{1}{\kappa^{m+1}\alpha}}+
      \abs{Q}^{\frac{1}{\kappa^{m+1}\alpha}}\\
    &\coex[\eqref{sss4}]{\le}
      \left\{\int_{Q} \left(v^{\frac{1}{s_{m}}} \right)^{\kappa^{m+1}\alpha}
      dz\right\}^{\frac{1}{\kappa^{m+1}\alpha}}  +
      \abs{Q}^{\frac{1}{\kappa^{m+1}\alpha}}\\
    &\coex[\eqref{ff}]{\le}
      c_{\alpha}^{\frac{1}{\kappa^{m}\alpha}}\left\{\int_{Q} \left(v^{\frac{1}{s_{m}}} \right)^{\kappa^{m}\alpha}dz\right\}^{\frac{1}{\kappa^{m}\alpha}}+c_{\alpha}^{\frac{1}{\kappa^{m}\alpha}}+
      \abs{Q}^{\frac{1}{\kappa^{m+1}\alpha}}\\
    &\coex[\eqref{sss4}]{\le}
      c_{\alpha}^{\frac{1}{\kappa^{m}\alpha}}\left\{\int_{Q}(w+k_{0})^{\kappa^{m}\alpha}dz\right\}^{\frac{1}{\kappa^{m}\alpha}}+
      c_{\alpha}^{\frac{1}{\kappa^{m}\alpha}}+
      \abs{Q}^{\frac{1}{\kappa^{m+1}\alpha}}\\
    &\coex[\text{Minkowski}]{\le}
      c_{\alpha}^{\frac{1}{\kappa^{m}\alpha}}\left\{\int_{Q}w^{\kappa^{m}\alpha}dz\right\}^{\frac{1}{\kappa^{m}\alpha}}+
      c_{\alpha}^{\frac{1}{\kappa^{m}\alpha}} \left(k_{0}\abs{Q}^{\frac{1}{\kappa^{m}\alpha}}
      +1 \right)+ \abs{Q}^{\frac{1}{\kappa^{m+1}\alpha}}\\
    &\coex[\kappa\ge 1]{\le}
      k_{1}\left\{\int_{Q}w^{\kappa^{m}\alpha}dz\right\}^{\frac{1}{\kappa^{m}\alpha}}+k_{1}\hk\forall~m\le m_{\alpha}, 
  \end{split}
\end{align*}
where $k_{1}\coloneqq\left(c_{\alpha}+ (k_{0}+1)(1+\abs{Q})+2
\right)^{\frac{2}{\alpha}}$.

By the mathematical induction, we obtain
\begin{align}
\label{eq:4}
  \begin{split}
    \left\{\int_{Q}w ^{\left( \kappa^{m_{\alpha}} \right)\alpha}
    dz\right\}^{\frac{1}{\left( \kappa^{m_{\alpha}} \right)\alpha}}
    &\coex{\le}
  k_{1}^{m_{\alpha}}\left\{\int_{Q}w^{\alpha}dz\right\}^{\frac{1}{\alpha}}+
      \sum_{i=1}^{m_{\alpha}}k_{1}^{i}\\
    &\coex[w = u^{+}]{\le}
  k_{1}^{m_{\alpha}}\left\{\int_{Q}\abs{u}^{\alpha}dz\right\}^{\frac{1}{\alpha}}+
      \sum_{i=1}^{m_{\alpha}}k_{1}^{i}.
  \end{split}
\end{align}

From \eqref{eq:3} and \eqref{eq:4}, we obtain
\begin{align*}
\norm{u^{+}}_{L^{\infty}(Q)}
 \le c_{1}^{\sum_{j=m_{\alpha}+1}^{\infty}\frac{1}{\kappa^{j}}}c_{2}^{\sum_{j=m_{\alpha}+1}^{\infty}\frac{j}{\kappa^{j}}}\max\left\{1,k_{1}^{m_{\alpha}}\left\{\int_{Q}\abs{u}^{\alpha}dz\right\}^{\frac{1}{\alpha}}+
  \sum_{i=1}^{m_{\alpha}}k_{1}^{i}\right\}.
\end{align*}
By a similar argument as above, we get the same bound for
$w=u^{-}$. The proposition is proved.
\end{description}
\end{proof}

Theorem \ref{th1} is a direct consequence of Proposition \ref{pro46}.

\begin{remark}
  Let  $\sigma_{N}\coloneqq\abs{B(0,1)}$ and $\beta \ge 2$. For $5 \le k \in
  \NN$, let $\{x_{5}, x_{6},\cdots\}$ be a dense subset of $\RR^{N}$ and
  \begin{align*}
    r_{k}&\coloneqq 2^{-k^{4k\beta}},\\
    v_{k}&\coloneqq\abs{B\left(x_{k},r_{k} \right)}=
           2^{-Nk^{4k\beta}}\sigma_{N},\\
    \overline{k}
    &\coloneqq
      \begin{cases}
        \frac{k+3}{2}, &\text{ if } k \text{ is odd},\\
        \frac{k+2}{2}, &\text{ if } k \text{ is even},
      \end{cases}\\
  \end{align*}
  On $\RR^{N}$, define 
  \begin{align*}
    b
    &\coloneqq 1,  \intertext{and}
    \overline{b}
    &\coloneqq   \chi_{\RR^{N} \backslash (\bigcup\limits_{k \ge 5}^\infty   B(x_{k},2r_{k})\backslash B(x_{k},r_{k}))}+   
\sum_{k=5}^{\infty}k^{4k}\chi_{B(x_{k},2r_{k})\backslash B(x_{k},r_{k}) }.
  \end{align*}
We have 
\begin{align*}
  \int_{B(x_{k},2r_{k})} k^{4\beta k}dz
  &\coex{=}  k^{4\beta k}2^{N}v_{k}
    \coex{=} 2^{N} \sigma_{N} k^{4\beta k}2^{-k^{4\beta
    k}}2^{-(N-1)k^{4\beta k}}  \\
  &\coex[N\ge 2]{\le}
    2^{N}\sigma_{N}2^{-k^{4\beta k}}
    \coex[N\ge 2,\beta\ge 2, k\ge 5]{\le} 2^{N}\sigma_{N}2^{-k},
\end{align*}
thus $\overline{b} \in L_{\text{loc}}^{\beta}(\RR^N)$.

To show that $\overline{b}$ is not a Muckenhoupt weight, it is enough
to show that $\overline{b} dz$ is not a doubling measure. To this end,
we calculate 
\begin{align*}
  \int_{B(x_{k},2r_{k})\setminus B(x_{k},r_{k})} k^{4k}dz
  &\coex{=} k^{4k}\sigma_{N}(2^{N(1-k^{4\beta k})}-2^{-Nk^{4\beta
    k}})\\
  &\coex{=} k^{4k}
    \sigma _{N}2^{-Nk^{4\beta k}}(2^{N}-1)  \\
  &\coex[N\ge 2]{>}3k^{4k}v_{k},\\
  \sum_{j=5}^{\overline{k}}\int_{B(x_{k},2r_{k})}j^{4j}dz
  &\coex[\overline{k}\le k-1]{\le}
    2^{N}\sum_{j=1}^{k-1}(k-1)^{4(k-1)}v_{k}
  \coex{\le} 2^{N}k^2k^{-4}k^{4k}v_{k}\\
  &\coex{\le} 2^{N}k^{-2}k^{4k}v_{k},\\
  \sum_{j=\overline{k}+1}^{k-1}\int_{B(x_{k},2r_{k})}j^{4j}dz
  &\coex{=} 2^{N} \sum_{j=\overline{k}+1}^{k-1}j^{4j}v_{k}
  \coex{\le} 2^{N}
    \left(\sum_{j=\overline{k}+1}^{k-1}j^{4j-4k} \right)k^{4k}v_{k}\\
  &\coex{\le} 2^{N} \left(\sum_{j=\overline{k}+1}^{k-1}j^{-4}
    \right)k^{4k}v_{k}\\
    &\coex{\le} 2^{N} \left\{\int_{\overline{k}}^{k}
    t^{-4}dt\right\} k^{4k}v_{k}
    \coex{\le} 2^{N} \overline{k}^{-3}k^{4k}v_{k} \\
  &\coex[\frac{k}{2}\le \overline{k}]{\le} 2^{N+3}k^{-3}k^{4k}v_{k},\\
  \sum_{j=k}^{\infty}\int_{B(x_{j},2r_{j})}j^{4j}dz
  &\coex{=} 2^{N}\sum_{j=k}^{\infty}j^{4j}2^{-Nj^{4\beta j}}\sigma_{N}\\
    &\coex{=}
      2^{N}\sum_{j=k}^{\infty}j^{4j}2^{-j^{4j}}2^{j^{4j}-Nj^{4\beta
      j}+Nk^{4\beta k}}v_{k}\\
  &\coex{=}
    2^{N}\sum_{i=0}^{\infty}(k+i)^{4(k+i)}2^{-(k+i)^{4(k+i)}}2^{(k+i)^{4(k+i)}-N(k+i)^{4\beta
    (k+i)}+Nk^{4\beta k}}v_{k}\\
  &\coex[\beta\ge 2]{\le}
    2^{N}\sum_{i=0}^{\infty}2^{(k+i)^{4(k+i)}-N(k+i)^{4\beta k}(k+i)^{4
    \beta i}+Nk^{4\beta k}}v_{k}\\
  &\coex{\le}
    2^{N}\sum_{i=0}^{\infty}2^{(k+i)^{4(k+i)}-4N(k+i)^{4\beta
    k}(k+i)^{4i}+Nk^{4\beta k}}v_{k}\\
  &\coex{\le}
    2^{N}\sum_{i=1}^{\infty}2^{-N(k+i)^{4\beta k}(k+i)^{4i}}v_{k}\\
  &\coex{\le} 2^{N}\sum_{i=0}^{\infty}2^{-i}v_{k}= 2^{N}v_{k}.
\end{align*}
Thus,
\begin{align*}
  \overline{b}(B(x_{k},2r_{k}))
  &\coloneqq \int_{B(x_{k},2r_{k})}\overline{b}dz \ge
    \int_{B(x_{k},2r_{k})\setminus B(x_{k},r_{k})} k^{4k}dz\ge
    3k^{4k}v_{k},\\
  \overline{b}(B(x_{k},r_{k}))
  &= \int_{B(x_{k},r_{k})} \overline{b}dz \\
 &\le  \sum_{j=5}^{k-1}\int_{B(x_{k},r_{k})} j^{4j}dz+
    \int_{B(x_{k},r_{k})}dz+
    \sum_{j=1}^{\infty}\int_{B(x_{k+j},2r_{k+j})} (k+j)^{4(k+j)}dz\\
  &\le 2^{N} \left(k^{-2}+2^{3}k^{-3}+ 2^{-N} k^{-4k}+
 k^{-4k} \right)k^{4k}v_{k},
\end{align*}
so 
\begin{align*}
\lim_{k\to\infty}
 \frac{\overline{b} \left(B(x_{k},2r_{k}) \right)}{\overline{b} \left(B(x_{k},r_{k}) \right)}\ge
 \lim_{k\to\infty}\frac{3}{2^{N}(k^{-2}+ 8k^{-3} +  2k^{-4k})}=\infty.
\end{align*}
In conclusion, $b^{-1} \in L^{\infty}(\RR^{N})$,
$\overline{b} \in L_{loc}^{\beta}(\RR^{N})$, and $\overline{b}$ is not
of class $A_{2}$.  However, $b\vert_{Q}$ and $\overline{b}\vert_{Q}$
satisfy conditions of Theorem \ref{th1} for every bounded open subset
$Q$ of $\RR^{N}$ and a sufficiently large $\beta$.
 \label{r4}
\end{remark}

\begin{remark} Let $\Omega\coloneqq B(0,1)$, $\overline{b}$ be as in
  Remark \ref{r4}, $\gamma\in (0,1)$ and
  $b(x)\coloneqq \dist(x,\partial\Omega)^{\gamma}$. Then $b$ and
  $\overline{b}$ satisfy condition  \eqref{eq:ccc} for sufficiently small $\gamma$ and sufficiently large
  $\beta$. Let $b_{ij}\coloneqq b_{i}\delta_{i}^{j}$ for every $i,j$
  in $\{1,\cdots,N\}$, $b_{1}=b$ and
  $b_{2}=\cdots= b_{N}=\overline{b}$. Applying our results, we obtain
  the boundedness for the solutions of a degenerate and non-uniform
  parabolic equation.
\label{r4b}
\end{remark}





\bibliographystyle{plain}
\bibliography{regularity}

\end{document}